\documentclass[reqno, 12pt]{amsart}     

\usepackage{xcolor,latexsym, amsfonts,amssymb,bbm,comment,
amsmath,cite, amsthm, bbold, float}
\usepackage{hyperref}
\usepackage{url}

\usepackage{dsfont}
\usepackage[T1]{fontenc}

\newtheorem{theorem}{Theorem}[section]
\newtheorem{corollary}[theorem]{Corollary}

\newtheorem{lemma}[theorem]{Lemma}
\newtheorem{proposition}[theorem]{Proposition}
 \newtheorem{definition}[theorem]{Definition}

\newtheorem{remark}[theorem]{Remark}
\newtheorem{example}[theorem]{Example}

\newtheorem{claim}[theorem]{Claim}
\newenvironment{claimproof}[1]{\par\noindent\underline{Proof of the claim:}\space#1}{\hfill $\blacksquare$}

\newcommand{\Lbrace}{\left \lbrace}
\newcommand{\Rbrace}{\right \rbrace}
\newcommand{\Pd}{\mathbb{P}}
\newcommand{\Ed}{\mathbb{E}}
\newcommand{\vertexset}[1]{\mathcal{V}_{#1}}
\newcommand{\dtv}[2]{\mathrm{d}_{\mathrm{TV}}\left(#1,#2\right)}
\newcommand{\name}{TBRW }

\newcommand{\degree}[2]{\mathrm{deg}_{#1}(#2)}
\newcommand{\tmix}{t_{\mathrm{mix}}}

\newcommand{\com}[1]{\textcolor{red}{\texttt{giulio:} #1}}

\newcommand{\comj}[1]{\textcolor{purple}{\texttt{janos:} #1}}

\newcommand{\myexp}{2(1-\gamma)+\delta}

\makeatletter
\@addtoreset{equation}{section}
\makeatother

\renewcommand\thetable{\thesection.\@arabic\c@table}

\bibdata{prob}
\bibstyle{alpha}

\title[Recurrence, transience and Power-law degree distribution in \name]{
Recurrence, transience and degree distribution for the Tree Builder Random Walk
}

\author{János Engl\"ander$^1$,  Giulio Iacobelli$^2$ and  Rodrigo Ribeiro$^3$}
\thanks{1. J. E.'s research was partially supported by Simons Foundation Grant 579110. }
\thanks{3. R.R. was supported
    by The Stochastic Models of Disordered and Complex Systems (NC120062), provided by the Millenium Scientific Initiative of the Ministry of Science and Technology (Chile). }

\address{
\newline
\newline
$^1$ Department of Mathematics, University of Colorado.
\newline
 Boulder, CO-80309-0395, USA.
\newline
e-mail: {\rm \texttt{janos.englander@colorado.edu}}
\newline
\newline
$^2$ Mathematical Institute,  Federal University of Rio de Janeiro (UFRJ).
\newline  Caixa Postal 68530, 21945-970, Rio de Janeiro, Brazil.
\newline
e-mail: {\rm \texttt{giulio@im.ufrj.br}}
\newline
\newline
$^3$ Department of Mathematics, University of Colorado. 
\newline
 Boulder, CO-80309-0395, USA.
 \newline 
 e-mail: {\rm \texttt{Rodrigo.Ribeiro@colorado.edu}}
}



\begin{document}

\maketitle

\begin{abstract}
We investigate a self-interacting random walk, whose dynamically evolving environment is a random tree built by the walker itself,  as it walks around.
At time $n=1,2,\dots$, right before stepping, the walker adds  a random number (possibly zero) $Z_n$ of leaves to its current position. We assume that the $Z_n$'s are independent, but, importantly, we do \emph{not} assume that they are identically distributed. 

We  obtain non-trivial conditions on their  distributions  under which the random walk is recurrent. 
This result is in contrast with some previous work in which, under the assumption that $Z_n\sim \mathsf{Ber}(p)$ (thus i.i.d.),  the random walk was shown to be  ballistic for every $p \in (0,1]$.  

We also obtain results on the transience of the walk, and the possibility that it ``gets stuck.''

From the perspective of the environment, we provide structural information about the sequence of random trees generated by the model when  $Z_n\sim \mathsf{Ber}(p_n)$, with $p_n=\Theta(n^{-\gamma})$ and $\gamma \in (2/3,1]$. 
 We prove that the empirical degree distribution of this random tree sequence converges almost surely to a power-law distribution of exponent $3$, thus revealing a connection to the well known preferential attachment model.
 %
 \end{abstract}

\section{Introduction}

The study of random walks on graphs that change over time has received increasingly more attention in the past decades, and has been the source of many new results in theoretical probability. Random walks on dynamic  graphs encompasses several   models: random walks in dynamic random environment \cite{avena2019random,hilario2015random,redig2013}, reinforced random walks \cite{pemantle2007survey,davis1990reinforced, angel2014localization,volkov2006phase, disertori2015transience} and excited random walk \cite{benjamini2003excited,zerner2005multi, berard2007, menshikov2012general,kosygina2012excited}. 

In all of these models, the underlying graph structure, more precisely the set of edges, is fixed and  the  graph dynamics reduces to  a change in time of the transition probabilities of the walker. In random walk in dynamic random environment, the change in time is  driven by a random process  independent of the walker dynamics, whereas  in   reinforced  and excited random walks (a.k.a., \emph{self-interacting} random walks), it  is coupled with the random walk trajectories. 


Recently, models of random walks which \emph{build their graph while walking} have been introduced, mutually coupling the random walk and the graph dynamics \cite{figueiredo2017building, iacobelli2019transient, iacobelli2019tree}. While they can be thought of as \emph{self-interacting} random walks, in a different way than in the predecessor models, they do not a-priori constrain the graph structure; the set of edges (as well as the vertex-set)  changes with time and strongly depends on the random walk trajectories (and vice versa).
In this paper, we study a model which belongs to this latter class; it is a particular case  of the Tree Builder  Random Walk (\name\!\!),  introduced in \cite{iacobelli2019tree}.

\subsection{Notation}
As usual, for two sequences of positive numbers,
$a_n=\mathcal{O}(b_n)$ will mean that $a_n/b_n$ remains bounded from above; $a_n=o(b_n)$ will mean that $a_n/b_n\to 0$ and $a_n=\Theta(b_n)$ will mean that $a_n/b_n$ remains bounded between two positive constants. By $X_n=\mathcal{O}(Y_n)$ for two sequences of random variables we mean that $\exists K>0$ for which $\limsup_n \frac{X_n}{Y_n}\le K$ a.s.; we use $o,\Theta$ for random variables in a similar manner.

For typographical reasons, we will often write $\mathbb{1}\{A\}$
instead of $\mathbb{1}_A$.
We recall that the \emph{order}  of a graph $G$ (denoted by $|G|$) is the cardinality of its vertex set, while the \emph{size} of a graph $G$  is the cardinality of its edge set. The set of vertices of a tree $T$ will be denoted by $\mathcal{V}(T)$ and then $|T|=|\mathcal{V}(T)|$= size of $T+1$. 

The Bernoulli distribution with parameter $p$ will be denoted by $\mathsf{Ber}(p)$, and $\rm d_{\rm TV}$ will denote the total variation distance between probability measures.

Finally, for us $\mathbb N$ will include zero, that is, $\mathbb N:=\{0,1,2,...\}$.

\subsection{The model}
 The model is parsimonious and depends on a sequence of probability laws~$\mathcal{L}:=L_1,L_2,...$, each $L_n$ supported on non-negative integers,  and a pair $(T_0, x_0)$, where  $T_0$ is a locally finite rooted tree with a self-loop\footnote{The role of the self-loop is to avoid periodicity.} attached at the root and $x_0$ is a vertex of $T_0$. The model is a stochastic process $\{(T_n, X_n)\}_{n\geq 0}$ on trees with a marked vertex (the current position of the walker), defined inductively. Given  $(T_n,X_n)$ we obtain~$(T_{n+1},X_{n+1})$ according to the rule below:  
\begin{enumerate}
	\item[(1)]\underline{Generate $T_{n+1}$}: create a non-negative random number of new vertices, independently of the history of the process up to time $n$, according to $L_n$ and connect them to~$X_n$; 
	\item[(2)] \underline{Obtain $X_{n+1}$}: given $T_{n+1}$,  choose uniformly (and independently from everything else) a neighbor of $X_n$ in $T_{n+1}$: this vertex will be $X_{n+1}$.
\end{enumerate}
Note that at every time $n$, first the tree $T_n$ may be modified (by the possible addition of new leaves) and then the random walk takes a step on the possibly modified tree $T_{n+1}$.

We refer to this model as $\mathcal{L}$-\name  to emphasize the  dependence on the sequence $\mathcal{L}:=\{L_n\}_{n\ge 1}$, which accounts for different probabilities of adding new vertices to the tree along the evolution. We denote by $\Pd_{x_0, T_0; \mathcal{L}}$ the law of $\{(T_n, X_n)\}_{n\geq 0}$ when $(T_0, X_0)=(T_0,x_0)$  and by $\Ed_{x_0, T_0;\mathcal{L}}$ the corresponding expectation.

It will be helpful to introduce also a  sequence of independent non-negative integer valued random variables $Z:=\{Z_n\}_{n\ge 1}$, such that $Z_n\sim L_n$. That is, $Z_n$ is the number of leaves added at time $n$.

We reserve the letters $P$ and $E$ for $P:=L_1\times L_2\times,\dots$ and for the  corresponding expectation.

Finally, $\mathcal{L}^{(m)}$ will denote the shifted sequence of laws $\mathcal{L}^{(m)}=\{L_{m+n}\}_{n\ge 1}$. When $\tau$ is an $\mathbb N$-valued stopping time with respect to the filtration generated by $Z_1,Z_2,...$, we will also use the notation $\mathcal{L}^{(\tau)}$ for the randomly shifted sequence $\mathcal{L}^{(\tau)}=\{L_{\tau+n}\}_{n\ge 1}$.

The behavior of the process $\mathcal{L}$-\name 
 may be studied from different perspectives. 
For example, one may look   at $\mathcal{L}$-\name{}  as a \textit{non-markovian}, \emph{self-interacting} random walk $\{X_n\}_{n\geq 0}$ whose environment is dynamically built  by the walker trajectories.  From this  first perspective,  understanding the dichotomy of \textit{transience/recurrence} and questions such as  \textit{ballisticity} and \textit{localization} are natural. 

Another interesting point of view consists of looking at  $\mathcal{L}$-\name as a random graph model. From this second  perspective questions concerning the \textit{structure} and \textit{degree distribution} of the random sequence of trees $\{T_n\}_{n\geq 0}$ stand out.

There is also a third perspective, which, in a way, is between the above two. The model $\mathcal{L}$-\name may be seen as a Markov chain $\{(T_n, X_n)\}_{n\geq 0}$, in the Polish space of locally finite rooted trees (see \cite{bordenave2012lecture} for an introduction to this space), i.e., each  pair $(T_n,X_n)$ may be interpreted as a tree $T_n$ rooted at $X_n$. From this 
perspective, the existence of stationary measures and the long-time behavior of the random rooted tree $(T_n,X_n)$, are typical  questions.  For this approach,  we refer the reader to \cite{figueiredo2017building} where the authors prove that, when $L_n = \mathsf{Ber}(p),\, \forall n\ge 1$ and $p\in (0,1]$ the sequence $\{(T_n,X_n)\}_{n\ge 0}$ converges, in a suitable sense, to a random infinite rooted tree. 

The model $\mathcal{L}$-\name belongs to a general  class of models of random walks that build their trees, called Tree Builder Random Walks (TBRW)~\cite{iacobelli2019tree}.  However, in \cite{iacobelli2019tree}, the authors assume a sort of \textit{uniform ellipticity} condition, namely, that $\inf_n P(Z_n\geq 1)=\kappa>0$,  i.e., that the probability of adding at least one new leaf is bounded away from zero. Under this  assumption, they prove that the corresponding random walk $\{X_n\}_{n\ge 0}$ is \textit{ballistic}.
 The $\mathcal{L}$-\name, with $L_n = \mathsf{Ber}(p),\, \forall n\ge 1$, $p \in (0,1]$, meets the uniform ellipticity condition, and in fact, the ballisticity of the walker had already been proven in \cite{figueiredo2017building}. 

Uniform ellipticity  is a key assumption because, in essence, it induces a \textit{regeneration structure}, similar in spirit to the one introduced in \cite{sznitman2001} for random walks in random environments on $\mathbb{Z}^d$, which allows the walker to forget fixed proportions of the space and regenerate the environment by starting  a completely ``new'' tree. 
 A natural question from  the random walk's perspective is:
 \begin{itemize}
     \item[{\bf (Q1)}] How does the random walk in $\mathcal{L}$-\name behave in the absence of the uniform ellipticity condition?   
 \end{itemize}
When looking  at  $\mathcal{L}$-\name as a random graph model, it is important to mention that random walks which build their graphs first appeared in the Network Science literature (see, \cite{saramaki2004scale, vazquez2003growing, cannings2013random, amorim2016growing} and references therein) as an attempt to  generate  \emph{scale-free} random graphs (graphs whose degree distribution is close to a power-law), while relaxing the assumption of \textit{global knowledge} present in the preferential attachment model of Barabási-Albert \cite{barabasi1999emergence} (global, here, refers to the fact the degree of every existing vertex must be known in order to decide the attachment probability). 
A natural question in this regard is:
\begin{itemize}
    \item[{\bf (Q2)}] Can $\mathcal{L}$-\name generate scale-free random graphs?
\end{itemize}
Note that, the ballisticity of the random walk in $\mathcal{L}$-\name, when $L_n= \mathsf{Ber}(p), p \in (0,1]$, is strongly intertwined with the  structure of the random trees $\{T_n\}_{n\geq 0}$. In particular, since the walker is moving away from its initial position fast, the trees generated  are \textit{path-like} (or vice versa!); at time $n$ the tree has a height of order $n$, and the degree distributions of $\{T_n\}_{n\geq 0}$ have  exponential tails.
This suggests that the scale-free nature of the sequence $\{T_n\}_{n\geq 0}$ for $\mathcal{L}$-\name may only emerge when the uniform ellipticity condition fails. 

\medskip 
Here we drop uniform ellipticity, by considering a sequence of laws $\{L_n\}_{n\ge 1}$ such that $\lim_{n\to\infty}L_n(\{0\}) \to 1$, and in such a case, we begin addressing  
the two questions above: we show the recurrence/transience of the random walk $\{X_n\}_{n\geq 0}$ under certain assumptions on $\mathcal{L}$, and also  the power-law degree distribution for the tree sequence $ \{T_n\}_{n\geq 0}$, when $L_n := \mathsf{Ber}(n^{-\gamma})$ with $\gamma \in (2/3,1]$. 


\begin{remark}[Strong Markov Property]
We will often use the Strong Markov Property for the Markov chain $\{(T_n, X_n)\}_{n\geq 0}$. Although it is quite obvious, it is worth noting that since time is discrete, this property automatically holds even though the state space is somewhat unusual (the space of marked/rooted trees).
\end{remark}

\subsection{Main results}
We want to define ``recurrence'' in such a way that when it holds, the walker visits any vertex infinitely often, with probability one.
In case of  ``transience'' we want the graph distance between the walker and any fixed vertex to tend to infinity almost surely. Importantly, when we say ``any vertex,'' we include even those that are eventually added. With these goals in mind, we make the following definition.
\begin{definition}[Transience and recurrence]
We say that the walker $X$ is 
     
{\sf(a)} \emph{ recurrent} if 
 $\forall m \in \mathbb{N}$, any finite tree $T$ and  $\forall x,v \in \mathcal{V}(T)$ (vertices of $T$),
\begin{equation*}
   \mathbb{P}_{T,x; \mathcal{L}^{(m)}}\left( X \text{ visits }v\text{ i.o.}\right) = 1\;,
\end{equation*}
that is, if starting at  the vertex $x$ of the tree $T$, the walker visits the vertex $v$ infinitely often a.s., when the tree growth is governed by  the  shifted sequence of laws $\mathcal{L}^{(m)}=\{L_{m+n}\}_{n\ge 1}$. 

{\sf(b)} \emph{transient} if $\forall m \in \mathbb{N}$, any finite tree $T$ and  $\forall x \in \mathcal{V}(T),
$\begin{equation*}
 \Pd_{T,x; \mathcal{L}^{(m)}}\left(\lim\limits_{n\to \infty} d(x,X_n)=\infty\right)=1\;, \end{equation*}where $d$ denotes the graph distance. 
\end{definition}
\begin{remark}\label{inter.pr} We mention two interesting problems regarding the above definitions.

(i) It is a non-trivial problem to check whether one could just as well define the same notion of recurrence (transience) by requiring that the relevant probability equals one for \emph{some} choice of the parameters, instead of \emph{all} choices.

(ii) Although, as we will see later,  in general, the walk can be neither recurrent nor transient, it would be nice to find conditions under which the dichotomy does hold. For example, we suspect that this is the case when each $Z_n$ is a Bernoulli variable.$\hfill\diamond$ 
\end{remark}
Define the expectations $m_n:=\sum_{i\ge 1}iL_n(i)$.
Our first results regard the behavior of the random walker in $\mathcal{L}$-\name. 
\begin{theorem}[Recurrent Regime] \label{thm:general_recurrence} Assume the following about $\mathcal{L}$:
\begin{enumerate}
\item[(A1)] $m_n<\infty, n\ge 1$;
\item[(A2)]  $q_n:=L_n(\{0\}) \nearrow 1$, as $n\to\infty$ (note monotonicity);
\item[(A3)] $(1-q_n)\cdot M_n^2\to 0$, as $n\to\infty$, where $M_n:=\sum_1^n m_k$;
\end{enumerate}
Then, the random walk in $\mathcal{L}$-\name is recurrent.
\end{theorem}
The heuristic behind condition  (A3) is as follows. The expected time until at least a new vertex is added after time $n$ is at most $1/(1-q_n)$ (due to condition (A2)) and  the average size of the tree at time $n$ is $M_n$.  Thus condition (A3) is a mixing condition to assure that the random walk has a chance to mix on the graph before adding a new vertex (since the cover time of a random walk on a graph of size $n$ is   $O(n^2)$). 
\begin{corollary}\label{cor:recurrence}
Let  $L_n:=\mathsf{Ber}(p_n)$, with  $p_n:=1-q_n:=n^{-\gamma}$. Then the  random walk in $\mathcal{L}$-\name is recurrent for $\gamma>2/3$.
\end{corollary}

Theorem~\ref{thm:general_recurrence} cannot be applied to the case $L_n = \mathsf{Ber}(p_n)$ with $p_n=\Theta(n^{-\gamma})$, for  $\gamma \leq 2/3$, since  assumption A3 is not met. However, in the specific situation $L_n = \mathsf{Ber}(p_n)$ with $p_n=\Theta(n^{-\gamma})$ we manage to extend recurrence of the walker  for all $\gamma>1/2$, as stated in the following theorem. 
\begin{theorem}[Recurrent Regime for $L_n = \mathsf{Ber}\left(n^{-\gamma}\right)$]\label{thm:rec} Consider a $\mathcal{L}$-\name where $L_n = \mathsf{Ber}\left(n^{-\gamma}\right)$ and $\gamma > 1/2$. Then, the walk is recurrent.
\end{theorem}
In Section \ref{tr.many} we will give conditions under which the walk is transient (see, Theorem \ref{thm:tr} and the corollary afterwards.) Specifically, we will see that when there are infinitely many growth times (i.e. $p_n$ is not summable),
and there are sufficiently many edges  grown at those times, the walk is never recurrent, and  under a mild condition on $p_n$ (Condition 3 of Theorem \ref{thm:tr}) it is transient.

In our next result, we focus on the tree structure rather than the walker, and  show that for $L_n = \mathsf{Ber}(p_n)$ with $p_n=\Theta(n^{-\gamma})$, in the regime 
$\gamma \in (2/3,1]$,
$\mathcal{L}$-\name generates  trees whose degree distributions  converge to the very same limiting distribution as in the celebrated Barabási-Albert model \cite{barabasi1999emergence,bollobas2001degree}. Note that by Corollary \ref{cor:recurrence}, in this regime recurrence holds.

Let $\vertexset{n}$ denote the vertex set of $T_n$; for $v \in \vertexset{n}$, $\degree{T_n}{v}$ 
denotes the degree of 
$v$ in $T_n$.
\begin{theorem}[Power-law degree distribution]\label{thm:powerlaw} Let $\{T_n\}_{n\geq 0}$ denote the sequence of random trees in $\mathcal{L}$-\name for $L_n: =\mathsf{Ber}(p_n)$ with $p_n=\Theta(n^{-\gamma})$. Then, for $\gamma \in (2/3, 1]$, any initial condition $(T_0,x_0)$ and $\forall d \in \mathbb{N}\setminus\{0\}$, it holds that
	\begin{equation*}
	\lim_{n \rightarrow \infty} \frac{1}{|\vertexset{n}|}\sum_{v\in \vertexset{n}}\mathbb{1}\{\degree{T_n}{v} = d\} = \frac{4}{d(d+1)(d+2)}\;, \quad \Pd_{T_0, x_0; \mathcal{L}}\text{-a.s.}
	\end{equation*}
\end{theorem}


\subsection{Open questions} In addition to Remark \ref{inter.pr},
we list below a couple of questions about the $\mathcal{L}$-\name which are yet to be answered.

\begin{itemize}
    \item[1)] For $L_n= \mathsf{Ber}(p_n)$ with $p_n=\Theta(n^{-\gamma})$ we  expect that for  $\gamma>0$ sufficiently small the random walk is transient. We conjecture that there is a phase transition for recurrence/transience according to whether  $\gamma>1/2$ or $\gamma<1/2$; is this really the case? What happens at $\gamma=1/2$? 
   %

    \item[2)] Concerning the tree structure, for $L_n= \mathsf{Ber}(p_n)$ with $p_n=\Theta(n^{-\gamma})$ and $\gamma>0$ sufficiently small we expect  a tree sequence with exponential tail degree distribution. Is $\gamma=1/2$ also a phase transition point for power-law/exponential tail degree distribution?
\end{itemize}



%
\section{Recurrence of the walker in \name  (proof of Theorem~\ref{thm:general_recurrence})}
In this section we  prove Theorem~\ref{thm:general_recurrence}, which states that the under certain assumptions on $\{Z_n\}_{n\geq 1}$ the random walk is recurrent. Without the loss of generality, throughout this section we may and will assume that our initial tree $T_0$ consists of just a single vertex, denoted by  $o$.  As will become clear along the proof, the initial tree plays no important role in the long-time behavior of the walker.

\medskip 

\begin{proof}[Proof of Theorem~\ref{thm:general_recurrence}] First note that if $\mathcal{L}$ satisfies (A1)-(A3) then, for any $m\ge 1$, the shifted sequence of laws $\mathcal{L}^{(m)}$  also satisfies those assumptions. This, in conjunction with the Markov property allows us to prove the theorem by only considering $m=0$.

Next, notice that by (A3), we can choose a sequence $g_n$ such that
    \begin{equation*}
        g_n \to \infty \text{ and } (1-q_n)\cdot g_n \cdot M_n^2 \to 0, \text{ as }n \to \infty\;.
    \end{equation*}
    Consider now the random variable
\begin{equation*}
    W_n = \sum_{k=n}^{n+g_nM_n^2} \mathbb{1}\{ Z_k \ge 1\}\;.
\end{equation*}
In words, $W_n$ counts how many times the walker has added at least one leaf during the time interval $[n, n+g_nM_n^2]$. By (A2), it follows that 
\begin{equation}\label{eq:mean}
    \Ed_{T_0, x ; \mathcal{L}}\left[W_n\right] = \sum_{k=n}^{n+g_nM_n^2} 1-q_k \le (1-q_{n})\cdot g_n \cdot M_n^2 \to 0\;,
\end{equation}
as $n$ tends to infinity. Now, let $A_n$ denote the  event
\begin{equation*}
    A_n = \{ \text{no grow in the interval }[n, n+g_nM_n^2]\}\;,
\end{equation*}
that is, the walker did not add any new leaves to the tree in the time interval $[n, n+g_nM_n^2]$.
By the bound in \eqref{eq:mean} along with the Markov inequality, we have that
\begin{equation}\label{eq:an}
    \Pd_{T_0, x; \mathcal{L}}(A_n^c) = \Pd_{T_0, x; \mathcal{L}}(W_n \ge 1) \le \Ed_{T_0, x; \mathcal{L}}\left[W_n\right] = o(1)\;.
\end{equation}
Finally, let $B_n$ denote the  event
\begin{equation*}
    B_n := \{ X \text{ visits the root some time in }[n, n+g_nM_n^2]\}\;.
\end{equation*}
Using the fact (see \cite{brightwellwinkler1990}) that the cover time of a tree of size $k$ is at most $2k^2$ gives us 
\begin{equation}\label{eq:bnan}
    \Pd_{T_0, x; \mathcal{L}}(B_n^c, A_n | V_n \le g_n^{1/4}M_n) \le \frac{\sqrt{g_n}M_n^2}{g_n M_n^2} = o(1)\;.
\end{equation}
By the simple Markov property we have that on the event $A_n$, $\{X_k\}_n^{n + g_nM_n^2}$ is distributed as a SSRW on $T_n$, which is a tree with at most $g_n^{1/4}M_n$ vertices. Moreover, being on $B_n^c$ means that this SRRW did not cover the entire graph in $g_n M_n^2$ steps. The bound then follows by the Markov inequality. Finally, 
\begin{equation*}
    \begin{split}
        \Pd_{T_0, x; \mathcal{L}}(B_n^c) & \le  \Pd_{T_0, x; \mathcal{L}}(B_n^c, A_n, V_n \le g_n^{1/4}M_n) + \Pd_{T_0, x; \mathcal{L}}(A_n^c) + \Pd_{T_0, x; \mathcal{L}}(V_n \ge g_n^{1/4}M_n) = o(1)\;,
    \end{split}
\end{equation*}
since the first term at the RHS is $o(1)$ by \eqref{eq:bnan}, the second one is $o(1)$ by \eqref{eq:an}, whereas the third one is $o(1)$ by the Markov inequality. The above inequality yields
$$
    \Pd_{T_0, x; \mathcal{L}}(X \text{ visits the root some time in }[n, n+g_nM_n^2]) = P(B_n) = 1-o(1)\;.
$$
Finally, by Fatou's lemma
$$
\Pd_{T_0, x; \mathcal{L}}(\limsup_{n\to \infty} B_n ) \ge \limsup_{n\to \infty}\Pd_{T_0, x; \mathcal{L}}(B_n) = 1\;,
$$
which proves that the root is visited infinitely many times for any initial condition $(T_0,x)$. To extend the result to a vertex which is eventually added to the tree, we apply the strong Markov property to the Markov process $\{(T_n,X_n)\}_{n\in \mathbb{N}}$ together with the result for a fixed initial condition. This proves that even those vertices which are added by the process in the future are visited infinitely often.
\end{proof}

\section{Recurrence of the walker in \name with $L_n = \mathsf{Ber}(p_n)$ with $p_n=\Theta(n^{-\gamma})$ (proof of Theorem~\ref{thm:rec})}

\newcommand{\comg}[1]{\textcolor{red}{\texttt{giulio:} #1}}

Let us begin recalling that the recurrence for $\gamma>2/3$ follows by Corollary~\ref{cor:recurrence}, thus we need to show it for $\gamma \in (1/2,2/3]$. 
Before we go to the proof of Theorem~\ref{thm:rec},  let us say some words about the general idea. The main idea behind the proof of Theorem~\ref{thm:general_recurrence} is that the walker $X$ most of the time mixes on $T_n$ before adding  new leaves. This allows us to rely on general bounds for cover time. However, mixing before adding new leaves is a strong condition which is not satisfied  for the sequence of laws  $L_n = \mathrm{Ber}(n^{-\gamma})$ with $\gamma \leq 2/3$. So, the strategy behind the proof of Theorem \ref{thm:rec} is to try to mix in a smaller tree. For this we must find a sequence of time intervals $[t_n, t_n + s_n]$ with the following characteristics:
\begin{enumerate}
\item $s_n$ is small enough so that the trees $T_{t_n}$ and $T_{t_n+s_n}$ are comparable in size;
\item $X$ spends a large enough amount of time on $T_{t_n}$ in the time interval $[t_n, t_n+s_n]$;
\item $s_n$ is large enough so that the time spent on $T_{t_n}$ is enough  for $X$ to to cover  $T_{t_n}$, although it may not  mix over $T_{t_n+s_n}$.
\end{enumerate}

 Given $ t, s \in \mathbb{N}$,  we will say that a vertex $v$ added between times $t$ and $t+s$ is {\it red} and define  $N_{t,t+s}$ as the number of visits to red vertices between times $t$ and $t+s$, i.e.,   
\begin{equation*}
    N_{t,t+s} := \sum_{j=t}^{t+s} \mathbb{1}\{X_j \not\in  V(T_t)\}\;.
\end{equation*}
The main ingredient in the proof of Theorem \ref{thm:rec} is the lemma below which provides the right order of the time window to observe $X$. 
Specifically, it states that, for some small~$\delta$, the sequence of time intervals $[n, n+  n^{2(1-\gamma)+\delta}]$ satisfies the properties $1)-3)$ above with  $t_n=n$ and $s_n=n^{2(1-\gamma)+\delta}$.
\begin{lemma}\label{lemma:nst} Consider a \name where $L_n = \mathrm{Ber}\left(n^{-\gamma}\right)$ and $\gamma \in (1/2,1]$. Then, for any initial condition $(T,x)$, any $m \in \mathbb{N}$ (time shift)  and $0<\delta < 2\gamma -1$, it holds that  %
\[
\mathbb{E}_{T,x; \mathcal{L}^{(m)}}\left[N_{n, n + n^{\myexp}}\right] = o\left(n^{\myexp}\right)\;.
\]
\end{lemma}

The proof of the above lemma is  given in  Subsection~\ref{ss:proofnst}.

\begin{proof}[Proof of Theorem \ref{thm:rec}] 
For a fixed $n$, let us denote by $S_{n, n+ n^{2(1-\gamma)+\delta}}$ the number of transitions of  the random walk $X$ on $T_n$, between time $n$ and $n+ n^{2(1-\gamma)+\delta}$, i.e., 
\[
S_{n, n+ n^{2(1-\gamma)+\delta}}:=\sum_{j=n+1}^{n+ n^{2(1-\gamma)+\delta}}
\mathbb{1}\{X_{j-1}\in V(T_n)\} \mathbb{1}\{X_j\in V(T_n)\}\;.
\]
Observe that if $S_{n,n + n^{\myexp}} \leq n^{\myexp}/2$, then $N_{n, n+ n^{\myexp}} \ge n^{\myexp}/4$. 
Thus, by 
Markov's inequality and Lemma~\ref{lemma:nst}, for $\delta<2\gamma -1$, we have that  
  \begin{equation*}
      \mathbb{P}_{T,x; \mathcal{L}^{(m)}} \left(S_{n, n+ n^{\myexp}} \le n^{\myexp}/2\right) = o(1)\;.
    \end{equation*}
Let $\widetilde{X}$ denote the walker $X$ seen only when it makes transitions over $T_n$. Specifically, let   $\phi_0\equiv n$ and, for $i\geq 1$, define recursively $\phi_i:= \inf\{\ell >\phi_{i-1}: X_\ell \in V(T_n)\}$. Now,  for $m\geq 0$,  consider the process $Y_m:=X_{\phi_m}$, which corresponds to the walker $X$ seen only when visits vertices in  $T_n$. Note that $Y$ is a lazy random walk on $T_n$ and may not be symmetric since the  probability of taking a self-loop depends on the red substructure dangling from the corresponding  vertex of $T_n$. Setting $\sigma_0\equiv 0$ and recursively for $j\geq 1$, $\sigma_j:= \inf\{m> \sigma_{j-1}: Y_m \neq Y_{\sigma_{j-1}}\}$, we can define $\widetilde{X}_k:=Y_{\sigma_k}$.  Note that  the process $\{\widetilde{X}_k\}_{k\geq 0}$ is a SSRW on $T_n$.  

Let $v \in V(T_n)$ and denote by  $A_n$ the following event:
  \begin{equation*}
        A_n := \left \lbrace X \text{ visits $v$ in the interval }[n, n+ n^{\myexp}] \right \rbrace\;.
    \end{equation*}
    Recall that $V_n:=\vert V(T_n)\vert$ is a sum of $n$ independent Bernoulli random variables for which it holds that 
    \begin{equation*}
   \mathbb{E}_{T,x; \mathcal{L}^{(m)}} \left[V_n \right] = \Theta\left(n^{1-\gamma} \right)\;.
    \end{equation*}
    Thus, by the above identity and Chernoff bounds there exist positive constants $C_1$ and $C_2$ depending on $|T|$ and $\gamma$ only such that
    \begin{equation*}
        \mathbb{P}_{T,x; \mathcal{L}^{(m)}} \left( V_n \ge C_1n^{1-\gamma}\right) \le e^{-C_2n^{1-\gamma}}\;.
    \end{equation*}
    
    On the event $\{V_n \le C_1n^{1-\gamma}\}$, $\widetilde{X}$ is a SSRW on a tree with at most $C_1n^{1-\gamma}$ vertices. By Theorem~2 in \cite{brightwellwinkler1990}, which states that the expected cover time of a SRRW in a tree with $k$ vertices is at most $2k^2$, Markov's inequality  yields that
    \begin{equation*}
    \mathbb{P}_{T,x; \mathcal{L}^{(m)}} \left( A^c_n , S_{n, n+n^{\myexp}} \ge n^{\myexp}/2, V_n \le C_1n^{1-\gamma}\right) = o(1)\;.
    \end{equation*}
    Indeed, on the event $\{A_n^c,S_{n,n+ n^{\myexp}} \ge n^{\myexp}/2, V_n \le C_1n^{1-\gamma}\}$, $\widetilde{X}$ takes at least $n^{\myexp}/2$ steps in a tree with at most $C_1 n^{1-\gamma}$ vertices and did not cover it (specifically, it  did not visit the vertex $v$). So, on this event, the time to cover $T_n$ is greater than $n^{\myexp}/2$, whereas its expected value is $O(n^{2(1-\gamma)})$. From the above we deduce that 
    $$
    \mathbb{P}_{T,x; \mathcal{L}^{(m)}} (A_n^c) = o(1)\;.
    $$
 Finally, by reverse Fatou, it follows that
    $$
    \mathbb{P}_{T,x; \mathcal{L}^{(m)}}\left( \limsup_n A_n\right) \ge \limsup_n \mathbb{P}_{T,x; \mathcal{L}^{(m)}} (A_n) = 1\;,
    $$
    which proves that the vertex $v$ of $T_n$ is visited infinitely often with probability one. 
    So, the above actually proves that the walker visits any vertex i.o., even those eventually added by the process.
\end{proof}

\subsection{Proof of Lemma \ref{lemma:nst}}\label{ss:proofnst}
 Lemma \ref{lemma:nst} states that in the time window $[n, n+n^{\myexp}]$ the walker spends $o(n^{\myexp})$ steps on the red vertices -- the ones added in the same time window. The natural direction to prove such result would be to prove that the expected number of visits to each new vertex is small enough and then sum over the random number of vertices we could add from $n$ to $n+n^{\myexp}$. However, the expected number of visits to a vertex is sensitive to its degree, which, in our model, can increase over time.
So,  our proof consists of the following two steps:

\begin{itemize}
    \item[$(i)$] bounding from above  the expected number of visits to a red vertex in the time window $[n, n+n^{2(1-\gamma)+\delta}]$ by a factor times its expected degree at time $n + n^{\myexp}$ (see, Equation~\eqref{eq:step_i});
    \item[$(ii)$] controlling the evolution of the degree  by showing that there exists $d_0 = d(\gamma)$ such that is extremely unlikely that a red vertex reaches degree greater than $d_0$ (see, Proposition~\ref{lem:step_ii}). 
\end{itemize}

Before proving Lemma~\ref{lemma:nst} we introduce some instrumental notions and an auxiliary result which is a quantitative version of step $(ii)$.

For $k \in \mathbb{N}$ we let $v_k$ be the vertex possibly added at time $k$, depending on the value of $Z_k$ (if $Z_k=0$ the vertex $v_k$ is not added to the tree). 
For $t\geq k$, let $D_{k,t}$ denote the degree of vertex $v_k$ at time $t$, i.e., 
\begin{equation}\label{eq:deg}
    D_{k,t} := \begin{cases}
        \mathrm{deg}_t(v_k), & \text{ if }Z_k = 1; \\
        0, & \text{otherwise}\;.
    \end{cases}
\end{equation}
Note that  if a vertex $v_k$ is not added (i.e., $Z_k=0$) then $D_{k,t}=0$ for all $t$. 
As mentioned above (step $ii$)), the proof of  Lemma~\ref{lemma:nst} relies on controlling the evolution of the degree and this is formally stated in the proposition  below  (whose   quite-technical proof is deferred to Subsection~\ref{sec:step_ii}). 

\begin{proposition}
\label{lem:step_ii}
Let $\gamma \in (1/2,1]$ and $\delta<2\gamma -1$. 
Fix the natural numbers $m$ (time shift) and  $d$ (degree).  Then, there exists a positive constant~$C_1$ depending on $\gamma$,$|T|$ and $d$, and there exists $\varepsilon>0$ (depending on $\gamma$ and $\delta$, but not on $d$),  such that for all large $n$'s and all $k\ge n$,
\[
\mathbb{P}_{T,x;\mathcal{L}^{(m)}} (D_{k,n+ n^{\myexp}} \ge d)\le \frac{C_1}{k^\gamma n^{\varepsilon (d-1)}}.
\]
\end{proposition}
Given $k\in \mathbb{N}$ (vertex index), $t,s \in \mathbb{N}$ (times) with $k\geq t$   and $d\geq 1$ (degree),   we denote by  $N_{t,t+s}^{(d)}(v_k)$  the number of visits to $v_k$ when it has degree $d$ in the time interval $[t,t+s]$. Formally, 
\begin{equation}\label{eq:N}
    N_{t,t+s}^{(d)}(v_k) := \begin{cases}
        \sum_{j=t}^{t+s}\mathbb{1}\{X_j = v_k, D_{k,j}=d\}\;, & \text{ if }Z_k = 1;\\
        0\;, & \text{ otherwise}\;.
    \end{cases}
\end{equation}
%

\begin{proof}[Proof of Lemma \ref{lemma:nst}] 
Recall that $N_{n,n+n^{\myexp}}$ denotes the number of visits to red vertices from $n$ to $n+n^{\myexp}$. The latter can be written as 
\[
    N_{n,n+n^{\myexp}} = \sum_{k=n}^{n+n^{\myexp}} N_{n,n+n^{\myexp}}(v_k)\;,
\]
 where $N_{n,n+n^{\myexp}}(v_k)$ denotes the number of visits to the vertex $v_k$ in the interval $[n,n+n^{\myexp}]$.
Moreover, to avoid clutter, let us introduce the shorthand $D^{(\delta)}_{n,k}:=D_{k,n+n^{2(1-\gamma)+\delta}}$, i.e., the degree  of vertex $v_k$ (with $k\geq n$) at time $n+n^{2(1-\gamma)+\delta}$.
Recall that, by the definition in \eqref{eq:deg},  $D^{(\delta)}_{n,k}=Z_kD^{(\delta)}_{n,k}$. Moreover, since $p_n$ is decreasing, it follows that
\begin{equation*}
Z_k D^{(\delta)}_{n,k} \succ Z_k\left(1+ \mathsf{Bin}\left(N_{n,n+n^{\myexp}}(v_k), \frac{1}{(n+n^{\myexp} + m)^\gamma}\right)\right)\;,
\end{equation*}
    under $ \mathbb{P}_{T,x;\mathcal{L}^{(m)}}$. Here $\succ$ denotes stochastic domination and $\mathsf{Bin}$ the binomial distribution. The above stochastic domination implies that 
    \begin{equation}\label{eq:step_i}
    \begin{split}
      \mathbb{E}_{T,x;\mathcal{L}^{(m)}} \left[N_{n,n+n^{\myexp}}(v_k)\right] &\le (n+n^{\myexp} + m)^\gamma \; \mathbb{E}_{T,x;\mathcal{L}^{(m)}} \left[Z_k\left(D^{(\delta)}_{n,k} -1\right)\right]
      \\
      &\leq n^\gamma(1+o(1)) \,\mathbb{E}_{T,x;\mathcal{L}^{(m)}} \left[Z_k\left(D^{(\delta)}_{n,k} -1\right)\right]\;,
      \end{split}
    \end{equation}
    where, the last inequality holds since $\gamma>1/2$ and $\delta< 2\gamma -1$.
    In order to control the RHS  of \eqref{eq:step_i}, note  that 
    \begin{equation*}
        \begin{split}
      \mathbb{E}_{T,x;\mathcal{L}^{(m)}} \left[Z_k\left(D^{(\delta)}_{n,k} -1\right)\mathbb{1}\{2 \le D^{(\delta)}_{n,k}\le d_0\} \right]
      & \le d_0 \,\mathbb{P}_{T,x;\mathcal{L}^{(m)}}(D^{(\delta)}_{n,k}\ge 2) \;,
        \\        \mathbb{E}_{T,x;\mathcal{L}^{(m)}} \left[Z_k\left(D^{(\delta)}_{n,k} -1\right)\mathbb{1}\{ D^{(\delta)}_{n,k}\ge d_0\} \right] & \le n^{\myexp}\, \mathbb{P}_{T,x;\mathcal{L}^{(m)}}(D^{(\delta)}_{n,k}\ge d_0)\;, 
        \end{split}
    \end{equation*}
    where, the bottom  inequality follows from noticing that 
    $D_{n,k}^{(\delta)}$ is at most $n^{2(1-\gamma)+\delta}$ (since a vertex added after time $n$ can, by time $n+n^{\myexp}$, have at most $n^{\myexp}$ neighbors). Moreover, by Proposition~\ref{lem:step_ii}, we have that for every $d\geq 1$ there exists $\varepsilon \in (0, \gamma)$ (which does not depend on $d$) and a positive constant $C$ (depending on $\gamma$,  |T| and $d$) such that 
    \[
    \mathbb{P}_{T,x;\mathcal{L}^{(m)}}(D^{(\delta)}_{n,k}\ge d)\leq \frac{C}{k^{\gamma} n^{\varepsilon(d-1)}}\;,
    \]
    which implies 
    \[
    \mathbb{E}_{T,x;\mathcal{L}^{(m)}} \left[Z_k\left(D^{(\delta)}_{n,k} -1\right)\right] \leq d_0\;  \frac{C_1}{k^{\gamma}n^\varepsilon} \; + \; n^{\myexp} \;  \frac{C_2}{k^{\gamma}n^{\varepsilon(d_0 -1)}}\;.
    \]
    Thus,  
    \begin{align*}
    \mathbb{E}_{T,x; \mathcal{L}^{(m)}}\left[N_{n, n + n^{\myexp}}\right] &\leq  n^{\gamma} (1+o(1)) \left(C_1 n^{-\varepsilon} d_0  + C_2 \frac{n^{\myexp}}{ n^{\varepsilon (d_0-1)}}\right) \sum_{k=n}^{n+n^{\myexp}}k^{-\gamma}
    \\
    &\leq (1+o(1))\left(C_1 n^{-\varepsilon} d_0  + C_2 \frac{n^{\myexp}}{ n^{\varepsilon (d_0-1)}}\right) n^{\myexp}\;,
    \end{align*}
    where, in the last inequality we use the trivial bound  $\sum_{k=n}^{n+n^{\myexp}} \frac{1}{k^{\gamma}} \le n^{\myexp -\gamma}$. 

    Since  $\varepsilon>0$  we clearly have that $
    \frac{n^{\myexp}}{n^\varepsilon}=o(n^{\myexp})$, while choosing $d_0$   large enough such that
    $\varepsilon(d_0-1) > 2(1-\gamma) + \delta$ implies that
    \begin{equation*}
    \frac{n^{ 2(1-\gamma)+\delta + 2(1-\gamma) +\delta}}{n^{\varepsilon(d_0-1)}} = o(n^{\myexp}),
    \end{equation*}  
    and we are done.\end{proof}

\subsubsection{Proof of Proposition~\ref{lem:step_ii}} \label{sec:step_ii}

The proof of  Proposition~\ref{lem:step_ii} relies on three  auxiliary results, namely Lemmas \ref{lem:fixedtree}, ~\ref{lemma:expecnst} and ~\ref{lemma:rec},  which will be presented below. 
 
The first lemma, a general result for SSRW on {\em fixed} trees,  will provide a basic step for the estimation of the dynamically changing trees later. 
\begin{lemma}\label{lem:fixedtree} Let $T$ be a finite tree with $|T|\ge 2$ and $v\in\mathcal{V}(T)$. Denote by  $N_t(v)$  the number of returns to $v$ in $t$ steps of a SSRW on~$T$. Then the following bound holds:
\[
     E_v\left[N_t(v)\right] \le \mathsf{deg}_T(v)\left[\frac{t+3}{2(|T|-1)}+24|T|\right]\;,
\]
where the expectation $E$ corresponds to the SSRW.
\end{lemma}

\begin{proof} We will make use of the Renewal Theorem where $R$ will denote the renewal function, corresponding to the random variable $Z_1\ge 0$. By formula (2) in\cite{Stone1972}, the renewal function $R$ satisfies the bound
\begin{equation}\label{eq: Stone.bound}
R(s)\le \frac{s}{\mu}+3\, \frac{\mu_{2}}{\mu^{2}},\ \forall s\ge 0,
\end{equation}
where $\mu=\mu_1$ and $\mu_i$ denotes the $i$-th moment ($i=1,2$) of  $Z_1$. 

In our case $Z_1$ will be the first return time $\tau_v$ and $R(t)=E_v \left[N_t(v)\right]$. To estimate $\mu_2,$ we are going  to use Theorems 4.1 and 5.1 in \cite{Moon1973}. Using our notation, they say that, when $v$ is {\it considered the root of the tree} and $M_i(v):=E_v\left[\tau_v^i\right]\ i=1,2$, one has,
\begin{equation}\label{eq: firstmomreturn}
   \mu= M_1(v)=\frac{2(|T|-1)}{\mathsf{deg}_T(v)};
\qquad
M_2(v)-M_1(v)=\frac{8}{\mathsf{deg}_T(v)}\sum_{y\neq v}|T_y|(|T_y|-1),
\end{equation}
where $T_y$ denotes the sub-tree of $y$, relative to the root $v$. (That is, $T_y$ is the tree formed by those vertices, for which the unique path between them and the root must include $y$.)
Hence,
\begin{equation*}
\frac{M_{2}(v)}{(M_{1}(v))^{2}}=\frac{1}{M_{1}(v)}+\frac{8}{\mathsf{deg}_T(v)(M_{1}(v))^{2}}\sum_{y\neq v}|T_y|(|T_y|-1).
\end{equation*}
Thus, by \eqref{eq: firstmomreturn},
\begin{equation*}
    \frac{M_{2}(v)}{(M_{1}(v))^{2}}\le\frac{\mathsf{deg}_T(v)}{2(|T|-1)}+\frac{8\,\mathsf{deg}_T(v)}{4(|T|-1)^{2}}\sum_{y\neq v}|T_y|(|T_y|-1),
\end{equation*}
and by using that $|T_y|\le |T|$, one obtains that


\begin{equation}\label{eq: withcube}
    \ \frac{M_{2}(v)}{(M_{1}(v))^{2}}\le\frac{\mathsf{deg}_T(v)}{2(|T|-1)}+\frac{2\,\mathsf{deg}_T(v)}{(|T|-1)^{2}} |T|^3.
\end{equation}
Since $|T|\ge 2$  and $x\ge 1$ implies $\frac{x}{x-1}\le 2$,
we may bound \eqref{eq: withcube} from above by
\begin{equation*}
    \frac{\mathsf{deg}_T(v)}{2(|T|-1)}+8\,\mathsf{deg}_T(v)\, |T|.
\end{equation*}
Plugging this back into \eqref{eq: Stone.bound}, we obtain that
\begin{equation*}
    \frac{R(t)}{t}\le \frac{1}{\mu}+\frac{3}{t}\left(\frac{\mathsf{deg}_T(v)}{2(|T|-1)}+8\,\mathsf{deg}_T(v)\, |T|\right),
\end{equation*}
and we are done by writing $E_v \left[N_t(v)\right]$ in place of $R(t)$ and using \eqref{eq: firstmomreturn}.
\end{proof}

 Since in \name the domain of the walker changes as it walks, we will need to keep track of the random times when new vertices are added and the times when they have their degrees increased. For this reason, we will need some more definitions.
 
For $d \in \mathbb{N}$, define inductively the following sequence of stopping times 
\begin{equation*}
    \eta_{k,1} := \begin{cases}
        k, & \text{ if }Z_k = 1; \\
        +\infty, & \text{ otherwise}\;,
    \end{cases}
\end{equation*}
and for $d\ge 2$,
\[
\eta_{k,d} :=  \begin{cases}
\inf \{ t > \eta_{k,d-1} \; : \; D_{k,t} = d\}\;, & \text{ if $\eta_{k,d-1} <\infty$;
}
\\
+\infty\;, & \text{ otherwise}\;.
\end{cases}
\]
 So, $\eta_{k,d}$ is the first time $v_k$ reaches degree $d$ (if $v_k$ is not added then $\eta_{k,d}=+\infty$, for all $d\geq 1$). 

Recall  that, given $k\in \mathbb{N}$ (vertex index), $t,s \in \mathbb{N}$ (times) with $k \geq t$ (we are interested in vertices possibly added after time $t$) and $d\geq 1$ (degree),     $N_{t,t+s}^{(d)}(v_k)$ denotes  the number of visits to $v_k$ when it has degree $d$ in the time interval $[t,t+s]$. 
Regarding $N_{t,t+s}^{(d)}(v_k)$ we have the following result: 
\begin{lemma}\label{lemma:expecnst} Fix some natural numbers  $m$ (time shift), $k$ (vertex index), $d$ (degree). Then there exist positive constants $C_1$ and $C_2$ depending on $\gamma$ only such that
    \begin{align*}
    \mathbb{E}_{T,x;\mathcal{L}^{(m)}} \left[N_{t,t+s}^{(d)}(v_k)\right] \le  C_1 d&\left(\frac{s}{t^{1-\gamma} -1  }{+ (|T|+t+s)^{1-\gamma}}\right)\mathbb{P}_{T,x;\mathcal{L}^{(m)}}\left( \eta_{k,d} < t+s\right) 
    \\
    &{ + 27 d(|T|+t+s)} \exp\left\{-C_2t^{1-\gamma}\right\}\;,
    \end{align*}
    $\forall t\ge k,\forall s\ge 0$.
\end{lemma}
\begin{proof} By the definition of $N_{t,t+s}^{(d)}(v_k)$ (see, \eqref{eq:N}) we have that 
$$N_{t,t+s}^{(d)}(v_k) = N_{t,t+s}^{(d)}(v_k)\mathbb{1}\{\eta_{k,d} < t+s\}.$$ 
Moreover,
%
     \begin{align*}
      \mathbb{1}\{\eta_{k,d} < t+s\} N_{t,t+s}^{(d)}(v_k) &\le \mathbb{1}\{\eta_{k,d} < t+s\} N_{\eta_{k,d},\eta_{k,d}+s}^{(d)}(v_k)\\
      &=\mathbb{1}\{\eta_{k,d} < t+s\} N_{0,s}^{(d)}(v_k)\circ \theta_{\eta_{k,d}}\;, \quad \mathbb{P}_{T,x;\mathcal{L}^{(m)}}\text{-a.s.}
    \end{align*}
    (This is because on the right-hand side, we only start counting the visits when order $d$ has already been reached.)   It is important to point out that $v_k$ in $N_{0,s}^{(d)}(v_k)$ is not the $k$-th vertex added by the shifted process, but the vertex $v_k$ which belongs to $T_{\eta_{k,d}}$ on the event $\{\eta_{k,d} < t+s\}$. Thus, by the Strong Markov Property it follows that 
    \begin{equation}\label{eq:smarkovnst}
      \mathbb{E}_{T,x;\mathcal{L}^{(m)}} \left[N_{t,t+s}^{(d)}(v_k)\right] \le \mathbb{E}_{T,x;\mathcal{L}^{(m)}}\left[ \mathbb{E}_{T_{\eta_{k,d}}, X_{\eta_{k,d}}; \mathcal{L}^{(m+\eta_{k,d})}} \left[N_{0,s}^{(d)}(v_k)\right]\mathbb{1}\{\eta_{k,d} < t+s\} \right]\;.
    \end{equation}
To bound the term
    \begin{equation}\label{handle}
    (*):=\mathbb{E}_{T_{\eta_{k,d}}, X_{\eta_{k,d}}; \mathcal{L}^{(m+\eta_{k,d})}} \left[N_{0,s}^{(d)}(v_k)\right] \mathbb{1}\{\eta_{k,d} < t+s\}\;,
    \end{equation} 
    we use coupling.
    To this end, 
    given $T_{\eta_{k,d}}$ and $X_{\eta_{k,d}}$, let $(W,P_{X_{\eta_{k,d}}})$ denote a simple random walk on $T_{\eta_{k,d}}$, starting from $X_{\eta_{k,d}}$, and  let $N^{W}_{s}$ be the number of returns to $v_k$ by the walker $W$ in $s$ steps.
     By ignoring some parts of $X$, we are going to couple  
    $\{X_i\}_{0\le i\le s}$  with $\{W_i\}_{0\le i\le s}$    (setting $X_0=W_0=X_{\eta_{k,d}}$) in such a way that  all possible subsequent visits of $X$ to $T_{\eta_{k,d}}(v_k)$ counted by  $N_{0,s}^{(d)}(v_k)$ are counted by $N^{W}_{s}$ as well.
   To do this, notice that 
   \\
   (i) $T_{\eta_{k,d}}(v_k)$ may reach order $d+1$ before completing $s$ steps, and hence the count in  $N_{0,s}^{(d)}(v_k)$ stops,
   \\
   (ii) $X$ may leave $T_{\eta_{k,d}}$ at other parts of the tree for ``excursions'' on the larger tree. 
   \\
   In  case (ii) we just delete those excursions outside of $T_{\eta_{k,d}}$, while in  case (ii) we must stop the walk $W$. Consequently, the walk $(W,P_{X_{\eta_{k,d}}})$  on $T_{\eta_{k,d}}$ obtained this way,   may take {\it less} than $s$ steps. To make up for that, and have exactly $s$ steps for $W$,  simply ``complete'' the remaining steps of  $W$ by letting it walking on $T_{\eta_{k,d}}$, independently of $X$. 
    It then follows that, given $T_{\eta_{k,d}}$ and $X_{\eta_{k,d}}$,
   one has
    \begin{equation}\label{itfollowsthat}
      (*)\le E_{X_{\eta_{k,d}}} \left[ N^{W}_{s} (X_{\eta_{k,d}})\right]+1\;.\
    \end{equation}
    (The reason we have to add one, is because we only count returns and exclude the original time.)
  Now, in order to bound the expected value in \eqref{itfollowsthat}, we exploit Lemma \ref{lem:fixedtree}. 
  It follows that, given $T_{\eta_{k,d}}$ and $X_{\eta_{k,d}}=v_k$,
 %
\begin{align}\label{eq: boundwith25}
(*) \le E_{v_k} \left[ N^{W}_{s}(v_k) \right] + 1 \le \frac{d\cdot  4s}{2(|T_{\eta_{k,d}}| -1)} +  25 d |T_{\eta_{k,d}}| 
      \le \frac{d\cdot 2s}{|T_t| -1} +  25 d |T_{t+s}|\;,
    \end{align}
    where, the last inequality holds since the assumption $k\geq t$ assures that   $T_{\eta_{k,d}}$ contains $T_t$,  and on the event $\{\eta_{k,d} < t+s\}$, one has $|T_{\eta_{k,d}}| \le |T_{t+s}|$. Plugging this into \eqref{eq:smarkovnst} leads us to 
    \begin{equation}\label{eq:expts}
      \mathbb{E}_{T,x;\mathcal{L}^{(m)}} \left[N_{t,t+s}^{(d)}(v_k)\right] \le d\cdot\mathbb{E}_{T,x;\mathcal{L}^{(m)}}\left[ \left( \frac{ 2 s}{|T_t| -1}  + 25 |T_{t+s}| \right) \mathbb{1}\{\eta_{k,d} < t+s\} \right]\;.
    \end{equation}
     Given the initial condition $(T,x)$, since the tree growth is governed by  the  shifted sequence of laws $\mathcal{L}^{(m)}=\{L_{m+n}\}_{n\ge 1}$,  it follows that, for every time $\hat{t}$,  
   we have
    \[
    |T_{\hat{t}}| =  |T| + \sum_{r=1}^{\hat{t}}Z_{r+m} \implies  \mathbb{E}_{T,x;\mathcal{L}^{(m)}} \left[ |T_{\hat{t}}| \right] 
    =|T|+\Theta \left(({\hat{t}}+m)^{1-\gamma}-m^{1-\gamma}\right)\;,
    \]
    and note that, since $\gamma\in (0,1)$, we have that 
    $(\hat{t}+m)^{1-\gamma}-m^{1-\gamma}\le \hat t^{1-\gamma}$ for all $m\ge 0$, 
    (and the constant involved in the $\Theta$ notation depends only on $\gamma$). Using that the $Z_i$ are independent Bernoullis, along with Chernoff bounds, there exist positive constants $C_2,C_2',C_3,C_3'$ depending on $\gamma$ only, such that 
    \begin{align*}
    &\mathbb{P}_{T,x;\mathcal{L}^{(m)}} \left( |T_t| \le C_2(|T| + t^{1-\gamma})\right) \le \exp\left\{-C_3t^{1-\gamma}\right\}; \\
&\mathbb{P}_{T,x;\mathcal{L}^{(m)}} \left( |T_{t+s}| \ge C_2'(|T| + (t+s)^{1-\gamma})\right) \le \exp\left\{-C_3't^{1-\gamma}\right\}\;. 
    \end{align*}
    Combining the above inequalities with \eqref{eq:expts} and using that 
 
    \begin{equation*}
        \begin{split}
        \mathbb{E}_{T,x;\mathcal{L}^{(m)}}\left[|T_{t+s}|\mathbb{1}\{\eta_{k,d} < t+s\} \right] & \le C'_2(|T| + (t+s)^{1-\gamma})\mathbb{P}_{T,x;\mathcal{L}^{(m)}}\left(\eta_{k,d}<t+s\right)  \\ 
        & \quad + (|T|+t+s)\exp\left\{-C'_3t^{1-\gamma}\right\}\;,
        \end{split}
    \end{equation*}    
    and 
        \begin{equation*}
        \begin{split}
        \mathbb{E}_{T,x;\mathcal{L}^{(m)}}\left[\frac{2s}{|T_{t}|-1}\mathbb{1}\{\eta_{k,d} < t+s\} \right] & \le \left(\frac{2s}{C_2(|T|+t^{1-\gamma})-1}\right)\mathbb{P}_{T,x;\mathcal{L}^{(m)}}\left(\eta_{k,d}<t+s\right)  \\ 
        & \quad + 2s\exp\left\{-C_3t^{1-\gamma}\right\}\;,
        \end{split}
    \end{equation*}   
    yields:
    \begin{align*}
&\mathbb{E}_{T,x;\mathcal{L}^{(m)}} \left[N_{t,t+s}^{(d)}(v_k)\right] \leq (2+25)d(|T|+t+s) \exp\left\{-{\min\{C_3,C_3'\}}t^{1-\gamma}\right\} 
    \\
    &+ \left(\frac{ { 2 }d \cdot  s}{C_2(|T|+t^{1-\gamma}) -1} { +25 d C'_2(|T| + (t+s)^{1-\gamma}) }\right)\mathbb{P}_{T,x;\mathcal{L}^{(m)}}\left( \eta_{k,d} < t+s\right)\;,
    \end{align*}
    which is enough to conclude the  result in the lemma. 
\end{proof}

The next result is a recursion for the tail probability of  $D_{k,t+s}$.

\begin{lemma}[Recursive bound]\label{lemma:rec} 
Fix natural numbers  $m$ (time shift), $k$ (vertex index), $d$ (degree), and fix  $\varepsilon < \gamma$. Then, there exist positive constants $C_1$ and $C_2$ depending on $\gamma$ such that for all $k\ge t$ and all $s\ge 0$,
\[
\mathbb{P}_{T,x;\mathcal{L}^{(m)}} (D_{k,t+s} \ge d+1)\le I+II\;,
\] 
where $I$ and $II$ are defined as follows. Using the shorthand 
{
\begin{equation}\label{eq: Delta}
        \Delta_d := \frac{C_1 d}{t^{\gamma - \varepsilon}}\left(\frac{ s}{t^{1-\gamma} -1 }+ (|T|+t+s)^{1-\gamma}\right),
\end{equation}}
\begin{align*}
          I:=&\, 
          \left\{\left(1+\frac{s}{t+m}\right)^{\gamma}
        \left[1 - \left( 1- \frac{1}{(t+s+m)^\gamma}\right)^{t^{\gamma-\varepsilon}}\right]+{\Delta_d}\right \} \mathbb{P}_{T,x;\mathcal{L}^{(m)}}
        \left( D_{k,t+s} \ge d\right);
         \\ II:=&\, \frac{{27d(|T|+t+s)}}{t^{\gamma - \varepsilon}}\exp\left\{-C_2t^{1-\gamma}\right\}\;.
          \end{align*}
\end{lemma}
\begin{proof} Let us begin by noticing  the  identity 
$  
\left \lbrace \eta_{k,d} \le t + s \right \rbrace = \left \lbrace D_{k,t+s} \ge d \right \rbrace,
$ 
and that
    \begin{equation}\label{eq:dktged}
        \begin{split}  \mathbb{P}_{T,x;\mathcal{L}^{(m)}} (D_{k,t+s} \ge d+1) & 
         \le \mathbb{P}_{T,x;\mathcal{L}^{(m)}} (D_{k,t+s} \ge d+1,  N_{t,t+s}^{(d)}(v_k) \le t^{\gamma - \varepsilon})\\ 
            & \quad + \mathbb{P}_{T,x;\mathcal{L}^{(m)}} ( N_{t,t+s}^{(d)}(v_k) \ge t^{\gamma - \varepsilon})\;.
        \end{split}
    \end{equation}
    By Markov's inequality and Lemma~\ref{lemma:expecnst} we have
    \begin{align}
    \label{eq:b1}
&\mathbb{P}_{T,x;\mathcal{L}^{(m)}} (  N_{t,t+s}^{(d)}(v_k) \ge t^{\gamma - \varepsilon}) \le  \Delta_d\mathbb{P}_{T,x;\mathcal{L}^{(m)}}\left( \eta_{k,d} < t+s\right) + \frac{ 27d(|T|+t+s)
     }
      {t^{\gamma-\varepsilon}} \exp\left\{
      -C_2t^{1-\gamma}
      \right\}\;.
    \end{align}
       For the first term of the RHS of \eqref{eq:dktged}, fix $j \le t^{\gamma-\varepsilon}$,  we then have the identity
    \begin{equation*}
        \left \lbrace D_{k,t+s} \ge d+1,  N_{t,t+s}^{(d)}(v_k) = j \right \rbrace = \left \lbrace  Z_{H'_j + 1} = 1, Z_{H'_{j-1}+1} = 0, \dots, Z_{H'_1+1} = 0, \eta_{k,d} < t+s \right \rbrace,
    \end{equation*}
    where $H'_i$ denotes the time of the $i$-th visit to $v_k$ after time $\eta_{k,d}$, i.e., after reaching degree $d$.
    In other words, the left-hand side of the above identity denotes the event in which $v_k$ has a degree at least $d+1$ before time $t+s$ and has been visited $j$ times while having a degree $d$. This means that at each of these $j$ visits to $v_k$, in the next step it has failed $j-1$ times to increase its degree to $d+1$ and only succeeds after the $j$-th visit, that is, at time $H'_j +1$. The failures and the success are described formally by the $Z_{H'_i+1}$'s. Using the independent nature of $Z$'s and the fact that $p_n$ is decreasing in $n$, we obtain the bound
    \begin{equation*}
        \begin{split}
\mathbb{P}_{T,x;\mathcal{L}^{(m)}} (D_{k,t+s} \ge d+1, &N_{t,t+s}^{(d)}(v_k) = j )
      \\
      &\le \frac{1}{(t+m)^\gamma}\left( 1 - \frac{1}{(t+s+m)^{\gamma}}\right)^{j-1}\mathbb{P}_{T,x;\mathcal{L}^{(m)}}\left( D_{k,t+s} \ge d\right)\;.
      \end{split}
    \end{equation*}
    Summing over $j$ from $1$ to $t^{\gamma - \varepsilon}$ leads to 
    \begin{align*}
       & \mathbb{P}_{T,x;\mathcal{L}^{(m)}} (D_{k,t+s} \ge d+1, N_{t,t+s}^{(d)}(v_k) \le t^{\gamma - \varepsilon}) \\
        &\qquad \le \frac{(t+s+m)^{\gamma}}{(t+m)^{\gamma}}\left[1 - \left( 1- \frac{1}{(t+s+m)^\gamma}\right)^{t^{\gamma-\varepsilon}}\right] \mathbb{P}_{T,x;\mathcal{L}^{(m)}}\left( D_{k,t+s} \ge d\right)\;,
    \end{align*}
which, combined with \eqref{eq:b1}, proves the result.
\end{proof}

We are finally  ready to prove Proposition~\ref{lem:step_ii}. 
\begin{proof}[Proof of Proposition~\ref{lem:step_ii}] 
 Set $t= n$ and $s=n^{\myexp}$, with $\delta<2\gamma -1$  and $\varepsilon<\gamma$ in Lemma~\ref{lemma:rec}, and recall the shorthand  $D^{(\delta)}_{n,k}=D_{k,n+n^{2(1-\gamma)+\delta}}$. Write $C'_1$ in place of $C_1$ in \eqref{eq: Delta} getting
           \begin{equation*}
               \begin{split}
                   \Delta_d  &=\\ &\!\!\frac{C'_1 d}{t^{\gamma - \varepsilon}}\left(\frac{ s}{(t^{1-\gamma} -1) }
                   + (|T|+t+s)^{1-\gamma}\right)
                   = \frac{C'_1d}{n^{\gamma - \varepsilon}}\left(\frac{ n^{\myexp}}{(n^{1-\gamma} -1) }+ (|T|+n+n^{\myexp})^{1-\gamma}\right)\;,
               \end{split}
           \end{equation*}
 and observe that since $\gamma\in(0,1)$ and we assumed that $\delta<2\gamma-1$, 
$$(|T|+n+n^{\myexp})^{1-\gamma}
\le |T|^{1-\gamma}+n^{1-\gamma}+n^{[2(1-\gamma)+\delta](1-\gamma)}\le |T|^{1-\gamma}+2n^{1-\gamma}\;,$$ hence
 \begin{equation}\label{eq: hence}
 \Delta_d \le \frac{C'_1d}{n^{\gamma - \varepsilon}}\left(\frac{ n^{\myexp}}{(n^{1-\gamma} -1) }+ |T|^{1-\gamma}+2n^{1-\gamma}\right)\le\frac{C_1d\cdot n^{\myexp}}{(n^{1-\gamma}-1)n^{\gamma-\varepsilon}} \;,
\end{equation}
for some positive constant $C_1$ depending on $\gamma$ and $|T|$ only.
By Lemma~\ref{lemma:rec} along with \eqref{eq: hence}, the probabilities $u(d):=\mathbb{P}_{T,x;\mathcal{L}^{(m)}} (D^{(\delta)}_{n,k} \ge d)$ satisfy that 
    \begin{equation*}
       \begin{split} 
        u(d+1)&\le 
      \left(\frac{ n+ n^{\myexp} + m}{n + m }\right)^\gamma\left(1 - \left(1-\frac{1}{(n+n^{\myexp} + m)^\gamma}\right)^{n^{\gamma -\varepsilon}} \right)  u(d)
      \\
      &\quad  +\frac{C_1dn^{\myexp}}{(n^{1-\gamma}-1)n^{\gamma-\varepsilon}}\,
      u(d)
    + \frac{{27d(|T|+n+ n^{\myexp})}}{n^{\gamma-\varepsilon}}\exp\left\{-C_2n^{1-\gamma}\right\}
    \\
      &\le \left(\frac{ 2n + m}{n + m }\right)^\gamma 
      \left(1 - \left(1-\frac{1}{(2n)^\gamma}\right)^{n^{\gamma -\varepsilon}} \right) u(d)\\
      &\quad +\frac{C_1dn^{\myexp}}{(n^{1-\gamma}-1)n^{\gamma -\varepsilon}}\,
      u(d)
    +\frac{{27d(|T|+n+ n^{\myexp})}}{n^{\gamma-\varepsilon}}e^{-C_2n^{1-\gamma}}\;.
       \end{split}
    \end{equation*}
    It follows that for all $n$  sufficiently large (recall that $1-\gamma\in (0,1/2)$ and $\delta<2\gamma -1$),  
\begin{align*}
       u(d+1) &\leq  \left(\frac{ C'_1}{n^{\varepsilon}}+\frac{C_1dn^{\myexp}}{(n^{1-\gamma}-1)n^{\gamma-\varepsilon}}\right) u(d)  + \frac{{27d(|T|+n+ n^{\myexp})}}{n^{\gamma-\varepsilon}}\exp\left\{-C_2n^{1-\gamma}\right\}\\
      &\leq  \left(\frac{ C'_1}{n^{\varepsilon}}+\frac{C_1dn^{\myexp}}{(n^{1-\gamma}-1)n^{\gamma-\varepsilon}}\right) u(d) + \frac{{27d(|T|+2n)}}{n^{\gamma-\varepsilon}}\exp\left\{-C_2n^{1-\gamma}\right\}  \\
      &\stackrel{n\ge |T|}{\le} \left(\frac{ C'_1}{n^{\varepsilon}}+\frac{C_1dn^{\myexp}}{(n^{1-\gamma}-1)n^{\gamma-\varepsilon}}\right) u(d) + 81dn^{1-\gamma+\varepsilon}\exp\left\{-C_2n^{1-\gamma}\right\}.
\end{align*}
Now, by choosing $\varepsilon$ so that 
    $$\theta:=
    1-\varepsilon -2(1-\gamma) - \delta > \varepsilon \iff \varepsilon < \frac12 (2\gamma-\delta-1)\in(0,\infty)\;,
    $$
it follows that for all $n$ sufficiently large,
\begin{align*}
  u(d+1)&\le \left(C'_1 n^{-\varepsilon}+C_1'' dn^{-\theta} \right)u(d) + d\exp\left\{-Kn^{1-\gamma}\right\}
  \\
  &
  \le \left(C'_1 +C_1'' d\right)n^{-\varepsilon}u(d) + d\exp\left\{-Kn^{1-\gamma}\right\}, 
\end{align*}
with some  constant $K>0$  depending on $C_2$ (which in turns depends on $\gamma$), and some $C_1''>0$ depending on $\gamma$ and $|T|$ only.

On the other hand, from the definition of $u(d)$, along with the above assumption that $2(1-\gamma)+\delta<1-2\varepsilon$, it is clear $u(d)=0$ unless $d\le 2n$  and $k\le n+n^{\myexp}$. Hence, we may and will only consider $d$'s (for a given $n$) that do not exceed $2n$  and $k\le n+n^{\myexp}$. 
It follows that for all $n\ge N_0(\gamma,d)\ge d/2$,
\begin{align}\label{eq:rec}
  u(d+1)\le \left(C'_1 +C_1'' d \right)n^{-\varepsilon}u(d) + \exp\left\{-Ln^{1-\gamma}\right\}  \;,
\end{align}
where $L:=K/2$. 
Continue the recursion \eqref{eq:rec} in $d$ backward all the way back to $u(1) = (k+m)^{-\gamma}\le k^{-\gamma}$, we obtain that for all $n\ge N_0(\gamma,d)\ge d/2$,
\begin{align*}
     u(d+1)&\le   \frac{1}{k^\gamma}\prod_{j=1}^{d} 
     \left(C_1' +C_1'' j\right)n^{-\varepsilon} + \exp\left\{-Ln^{1-\gamma}\right\}\sum_{j=0}^{d-1}
     \left(C_1' +C_1'' d\right)^j n^{-\varepsilon j}
     \\
      &    \le   
     \frac{1}{k^\gamma} 
     \left(C_1'+C_1'' d\right)^{d} n^{-d\varepsilon} + \exp\left\{-Ln^{1-\gamma}\right\}\sum_{j=0}^{d-1}
     \left(C_1' +C_1'' d\right)^j \\
     & = \left(C_1'+C_1'' d\right)^{d} n^{-d\varepsilon} + C_2\exp\left\{-Ln^{1-\gamma}\right\}\;,
\end{align*}
where $C_2:= \sum_{j=0}^{d-1}\left(C_1' +C_1'' d\right)^j$ is a constant depending on $\gamma, |T|$ and $d$ only. Moreover, since $k\le n+n^{\myexp}$, by replacing $\left(C_1'+C_1'' d\right)^{d}$ by a larger constant if necessary, there exists a positive constant $C_3$ depending on $\gamma, |T|$ and $d$ only, such that for all $k \in [n, n+n^{\myexp}]$ and $n\ge 1$
$$
C_2\exp\left\{-Ln^{1-\gamma}\right\} \le \frac{1}{k^\gamma} 
     C_3 n^{-d\varepsilon}\;,
$$
which implies that 
\begin{equation*}
    u(d+1) \le \frac{1}{k^\gamma} 
     C_3 n^{-d\varepsilon} + C_2\exp\left\{-Ln^{1-\gamma}\right\} \le  \frac{2C_3}{k^\gamma} n^{-d\varepsilon}\;.
\end{equation*}


Recalling that, by definition,
$u(d)=\mathbb{P}_{T,x;\mathcal{L}^{(m)}} (D^{(\delta)}_{n,k} \ge d)$,  the proof is complete.
\end{proof}

\section{Power-law degree distribution in \name trees (proof of Theorem \ref{thm:powerlaw})}

In this section, we prove  that the degree distribution of the random tree sequence $\{T_n\}_{n\geq 0}$ in the $\mathcal{L}$-\name with $L_n=\mathsf{Ber}(p_n)$, $p_n=\Theta(n^{-\gamma})$ and $\gamma \in (2/3,1]$ converges to a power-law distribution with exponent $3$, the very same limiting distribution as in the  Barabási-Albert model of \emph{preferential attachment} (PA) \cite{barabasi1999emergence,bollobas2001degree}. 

It will be useful to look at the random tree sequence~$\{T_n\}_{n\geq 0}$ only at the random times (``growth times'') when the \name adds new vertices to the tree. For this reason, we begin by recursively defining the following sequence of stopping times:
\begin{equation}\label{def:tau}
\tau_k := \inf \Lbrace n > \tau_{k-1} \, : \, Z_n = 1\Rbrace\;,
\end{equation}
where $\tau_0 \equiv 0 $. In words, $\tau_k$ is the time when the $k$-th (new) vertex is added. 

It will be useful to understand the asymptotic behavior of the stopping times $\{\tau_k \}_{k\ge 0}$. The following lemma, whose proof will be postponed to the end of this section, essentially tells us that under the regime $\gamma > 2/3$ we need to wait large amounts of time to see the walker adding another leaf.
\begin{lemma}[Growth times are rare]\label{lemma:limtauk} Consider a $\mathcal{L}$-\name, where $L_n = \mathsf{Ber}(p_n)$ with $p_n=\Theta(n^{-\gamma})$ and  $\gamma>2/3$. Then, there exists a sufficiently small $\delta > 0$, depending only on $\gamma$,  such that
	\begin{equation}\label{eq:tau-lemma}
	\lim_{k \rightarrow \infty} k^{2+\delta}\tau_{k}^{-\gamma} = 0\;, \quad  \Pd_{T_0,x; \mathcal{L}}\text{-a.s.}\;
	\end{equation}
\end{lemma}
For the rest of this proof fix $\delta>0$ such that \eqref{eq:tau-lemma} holds.
\begin{definition}[Good and bad time intervals]
With Lemma \ref{lemma:limtauk} in mind, we say that $\Delta \tau_k := \tau_k - \tau_{k-1}$ is \emph{good} if 
\begin{equation}\label{def:goodtau}
\Delta \tau_k \ge k^{2+\delta} +1\;,
\end{equation}
 and otherwise we say that it is \emph{bad}.
\end{definition}

To avoid clutter,   we will write $\widetilde{T}_k := T_{\tau_k}$, and $\widetilde{\vertexset{}}_k$ for the set of vertices of $\widetilde{T}_k$ and $\widetilde{\mathcal{F}}_k$ for $\mathcal{F}_{\tau_{k}}$. Notice that for any~$n \in (\tau_{k-1}, \tau_{k})$, $T_n$ and $\widetilde{T}_{k-1}$ has the same degree distribution. For this reason we focus our attention on the process $\{\widetilde{T}_k\}_k$.

Let us now look at the  evolution of the random graph model   $\{\widetilde{T}_k\}_k$ from a slightly different, although equivalent, perspective.   At any ``time'' step $k$, a new vertex $v_k$ and two half-edges,  $h_{k,1}$ and $h_{k,2}$,  are added to $\widetilde{T}_{k-1}$. The two half-edges account for the edge which connects the vertex $v_k$ to the tree $\widetilde{T}_{k-1}$. While the half-edge $h_{k,1}$  is always incident to the new vertex $v_k$, the half-edge $h_{k,2}$ will be attached  to the vertex in $\widetilde{\vertexset{}}_{k-1}$ where the random walk resides at time $\tau_k-1$, thus determining to which of the existing vertices the new vertex $v_k$ will be connected to in $\widetilde{T}_k$.

Leveraging on the above-mentioned   equivalent perspective of the evolution of $\{\widetilde{T}_k\}_k$,  we are going to introduce  a color-assignment process for  the half-edges which will be instrumental for proving the result. 
Specifically, at any time $k$, we color the two half-edges $h_{k,1}$ and $h_{k,2}$ in either \emph{blue} or \emph{red} according to the following rule: Let $B_{k-1}$ and $R_{k-1}$ denote the total number of half-edges blue and red  in $\widetilde{T}_{k-1}$, respectively (w.l.o.g, we assume $B_1=2, R_1=0$, which means that $T_0$ is a single edge whose half-edges are both blue). Then, for $k\ge 2$,
\begin{itemize}
	\item if $\Delta \tau_{k}$ is \textit{bad}, we color both $h_{k,1}$ and $h_{k,2}$ \textit{red};
	\item if $\Delta \tau_{k}$  is \textit{good}, we color $h_{k,1}$ \textit{blue}, while we flip a biased coin to decide the color of  $h_{k,2}$. The probability that $h_{k,2}$ is \textit{blue} is $B_{k-1}/2(k-1)$,  and 
	\textit{red} otherwise.
\end{itemize}
%
Note that $B_{k-1}/2(k-1)$ is the ratio of blue half-edges in $\widetilde{T}_{k-1}.$
Note also that the color assignment depends on the evolution of $\{\widetilde{T}_k\}_k$ through the conditions $\Delta\tau_{k}$ being bad or good, which in turns depends on $\gamma$ (see, Equation~\eqref{def:goodtau}), but it does not depend on the position of the random walk at time $\tau_k-1$.

As we shall soon show, in order to prove Theorem~\ref{thm:powerlaw}, it suffices to show that the empirical \emph{blue} degree distribution of $\widetilde{T}_k$  converges to~$4/d(d+1)(d+2)$. 
The first ingredient is to  show that the total number of \emph{red} half-edges is small compared to the total number of \emph{blue} half-edges. This is addressed in the next lemma, whose proof is deferred to the end of this section. 
\begin{lemma}[Not many red edges]\label{lemma:Rk} There exist $\varepsilon,\delta'>0$ depending on $\gamma$ only, such that for $k$ sufficiently large
\begin{equation*}
	\Pd_{T_0,x;\mathcal{L}} \left( R_k > k^{1-\varepsilon} \right) \le \exp\left\{-k^{\delta'}\right\}.
	\end{equation*}
\end{lemma}
Next, we introduce some notation which will be instrumental in proving Theorem~\ref{thm:powerlaw}. Let $b_k(v)$ (resp. $r_k(v)$) be the \textit{blue} (resp. \textit{red}) degree of a vertex~$v$ in $\widetilde{\vertexset{}}_k$, that is, $b_k(v)$ (resp.  $r_k(v)$) counts  the number of blue (resp. red) half-edges incident to $v$ in $\widetilde{T}_k$. Also, for a fixed $d\in \mathbb{N}\setminus \{0\}$, let~$B_k(d)$ denote the number of vertices in $\widetilde{T}_k$ whose blue degree is exactly $d$, i.e.,
\begin{equation*}
B_k(d) := \sum_{v\in \widetilde{\vertexset{}}_k}\mathbb{1}\{b_k(v) = d\}\;.
\end{equation*}
Moreover, we say that a vertex $v \in \widetilde{\vertexset{}}_k$ is \textit{blue} if all half-edges incident to it are \textit{blue}.

\medskip 
Our main argument will rely on  stochastic approximation techniques \cite{robbins1951stochastic, benaim1999dynamics} to deal with the error arising from the fact that 
the blue degree of a vertex does not evolve according to a ``pure'' linear preferential attachment scheme, i.e., the probability of a vertex increasing its blue degree by one  is not simply proportional to its blue degree, but there is an (additive) error which must be controlled. 
To keep the paper self-contained, we state below the specific stochastic approximation framework  we shall use, which is taken \textit{ad litteram} from \cite{dereich2014robust}. 

\begin{lemma}[Stoch. Approx.; Lemma~3.1 in  \cite{dereich2014robust}]\label{lem:approximation}
	Let $\{Q_n\}_{n\geq 0}$ be a non-negative stochastic process satisfying the following recursion:
	\[
	Q_{n} - Q_{n-1} = \frac{1}{n} \left( \Psi_{n-1} -Q_{n-1} \Phi_{n-1} \right) +  M_{n}- M_{n-1}\;,
	\]
	where, $\{\Psi_{n}\}_{n\geq 0}, \{\Phi_{n}\}_{n\geq 0}$ are almost-surely convergent processes with deterministic limits $\psi>0$ and $\phi>0$, respectively, and $\{M_{n}\}_{n\geq 0}$ is an almost-surely convergent process. Then, $\lim_{n \to \infty}Q_n= \psi/\phi$, almost surely.
\end{lemma}
Together with the stochastic approximation framework, we will need a lemma which formalizes the discussion made in the previous paragraph. Roughly speaking, it says that whenever $\Delta \tau_k$ is good, the position of $X_{\tau_k -1}$ is selected according to a mixture of a preferential attachment scheme and some random error. 

\begin{lemma}[Almost PA at good times]\label{lemma:xtaucond} There exists $\delta>0$ such that for every $k\in \mathbb{N}$ and $t>k^{2+\delta}$, 
    \begin{align*}
	&\Pd_{ T_{\tau_{k-1}}, X_{\tau_{k-1}}; \mathcal{L}^{(\tau_{k-1})}}\left( X_{\tau_1 -1 } = v\mid \tau_1 = t\right)\\
	&\qquad\qquad= \frac{\degree{\tau_{k-1}}{v}}{2(k-1)} +\mathsf{error}(v,k)\;, \text{for}\  v\in \widetilde{T}_{k-1}\;,
	\end{align*}
	where the term $\mathsf{error}(v,k)$ represents a $\mathcal{F}_{\tau_{k-1}}$-measurable  random variable bounded by $e^{-k^\delta}, \Pd_{T_0,x;\mathcal{L}}$-almost surely.
\end{lemma}
Again, we will defer the proof of the above lemma to the end of this section. Now we will show how Theorem \ref{thm:powerlaw} follows from the aforementioned lemmas.
\begin{proof} [Proof of Theorem \ref{thm:powerlaw}] We begin by noticing that, if $N_k(d)$ denotes the number of vertices of degree $d$ in $\widetilde{T}_k$ (ignoring the colors), then  for all~$d\in \mathbb{N}\setminus \{0\}$,
\begin{align}\label{first.easy}
    0\le B_k(d)-\sum_{v\in \widetilde{\vertexset{}}_k} \mathbb{1}\{b_k(v)=d, v\text{ is blue}\} \le R_k\;.
\end{align}
%
Note also that 
\begin{align}\label{second.easy}
    0\le \sum_{v\in \widetilde{\vertexset{}}_k}\mathbb{1}\{\degree{\widetilde{T}_k}{v} = d\}-\sum_{v\in \widetilde{\vertexset{}}_k} \mathbb{1}\{b_k(v)=d, v\text{ is blue}\}\le R_k\;.
\end{align}
By virtue of Lemma~\ref{lemma:Rk}, $\lim_k R_k/k=0,\ \Pd_{T_0,x;\mathcal{L}}$-a.s. Putting this together with \eqref{first.easy} and \eqref{second.easy}, we see that in order to prove  Theorem~\ref{thm:powerlaw}, it is enough to show the following claim.

\medskip 
\begin{claim}[Limit distribution for blue degrees]\label{Bk} For any  $d \in \mathbb{N}\setminus \{0\}$,
	\begin{equation*}
	\lim_{k \rightarrow \infty} \frac{B_k(d)}{k} = \frac{4}{d(d+1)(d+2)}\;, \quad  \Pd_{T_0,x; \mathcal{L}}\text{-a.s.}
	\end{equation*}
\end{claim}

%


\begin{claimproof} 
We want to show that the evolution of the blue degrees behaves much like a preferential attachment scheme. This is not a-priori clear because mixing on a long ``good'' time interval only yields\footnote{Since the stationary distribution on a fixed graph is proportional to the degrees.} that the full degrees (including the red degrees) of the vertices behave like that scheme. However, we are going  to show that the value of
\begin{equation*}
	\Pd_{T_0,x;\mathcal{L}}\left( \Delta b_k(v) = 1 \mid\Delta \tau_k \text{ is good}, \widetilde{\mathcal{F}}_{k-1}\right) \;,
	\end{equation*}
	is ``close'' to 
$\frac{b_{k-1}(v)}{2(k-1)}$, where here and in the sequel, $v\in \widetilde{\mathcal{V}}_{k-1}$. This means that $v$ receives the new blue half edge (if there is one) with a probability that is roughly proportional to its existing blue degree.

In order to achieve this, by Lemma \ref{lemma:xtaucond} and recalling that $\Delta \tau_k$ is good if $\Delta \tau_k \ge k^{2+\delta}+1$, we have that
	\begin{align}\label{eq:xvdeg}
	\Pd_{T_0,x;\mathcal{L}}\left( X_{\tau_{k} - 1}  = v  \mid \Delta \tau_k \text{ is good}, \mathcal{F}_{\tau_{k-1}}\right) =\frac{\degree{\tau_{k-1}}{v}}{2(k-1)} + \mathsf{error}(v,k)\;, \; \Pd_{T_0,x;\mathcal{L}}\text{-a.s.} 
	\end{align}
	Where $\mathsf{error}(v,k)$ represents a $\mathcal{F}_{\tau_{k-1}}$-measurable random variable bounded by $e^{-k^\delta}, \Pd_{T_0,x;\mathcal{L}}$-almost surely. %
	From the definition of the color-assignment process, we have that 
	$$
	\Pd_{T_0,x;\mathcal{L}}\left( \Delta b_k(v) = 1 \mid  X_{\tau_{k} - 1}  = v,\Delta \tau_k \text{ is good}, \widetilde{\mathcal{F}}_{k-1}\right) = \frac{B_{k-1}}{2(k-1)}\;,$$ 
	which, combined with  \eqref{eq:xvdeg}, and using that  $\degree{\tau_{k-1}}{v}=b_{k-1}(v) + r_{k-1}(v)$  and $B_k +R_k = 2k$, leads to
	\begin{equation}\label{eq:prob_increments}
	\begin{split}
	&\Pd_{T_0,x;\mathcal{L}}\left( \Delta b_k(v) = 1 \mid \Delta \tau_k \text{ is good}, \mathcal{F}_{\tau_{k-1}}\right)=
	\\
	&\qquad \frac{b_{k-1}(v)}{2(k-1)} + \underbrace{ \frac{B_{k-1}r_{k-1}(v) - b_{k-1}(v)R_{k-1}}{(2(k-1))^2} \pm \frac{B_{k-1}}{2(k-1)}\mathsf{error}(v,k)}_{:=F_{k-1}(v)}\;,
	\end{split}
	\end{equation}
	where $F_{k-1}(v)$ is a $\widetilde{\mathcal{F}}_{k-1}$-measurable random variable satisfying that
	\begin{equation}\label{eq:Forder}
	F_{k-1}(v) = \mathcal{O}\left( \frac{B_{k-1}r_{k-1}(v)-b_{k-1}(v)R_{k-1}}{(k-1)^2}\right)\;.
	\end{equation}
	Here $\mathcal{O}$ refers to a limit as $k\to\infty$. Note that although the collection of $v$'s to which this applies, keeps changing, the term $\mathsf{error}(v,k)$  is uniformly bounded in $v$.
	Denote by $\mathcal{B}_k(d)$ the subset 
	of vertices in $\widetilde{\vertexset{}}_k$ whose blue degree is exactly $d$;
   note that  $|\mathcal{B}_k(d)|={B}_k(d)$. Clearly,
   
	\begin{align*}
	\Delta B_k(d)=\begin{cases}
	\mathbb{1}\{\Delta \tau_{k} \text{ is good} \} + {\!\!\displaystyle\sum_{v \in \mathcal{B}_{k-1}(0)} \!\!\mathbb{1}\left\{ \Delta b_{k}(v)=1\right \}} - \!\!\displaystyle\sum_{v \in \mathcal{B}_{k-1}(1)} \!\!\mathbb{1}\left\{ \Delta b_{k}(v)=1\right \}  \;, & d=1\;,
	\\[7pt]
	\!\displaystyle\sum_{v \in \mathcal{B}_{k-1}(d-1)} \!\!\mathbb{1}\left\{ \Delta b_{k}(v)=1\right \}-\sum_{v \in \mathcal{B}_{k-1}(d)} \!\!\mathbb{1}\left\{ \Delta b_{k}(v)=1\right \}\;, & d>1\;.
	\end{cases}
	\end{align*}
     To avoid clutter, henceforth we set  
	$$
	W_{k-1}:= \Pd_{T_0,x;\mathcal{L}}\left(\Delta \tau_k \text{ is good}\mid  \mathcal{F}_{\tau_{k-1}}\right).
	$$ 
	Notice that for $v \in \mathcal{B}_{k-1}(0)$ Equation \eqref{eq:prob_increments} gives us that 
	$$
	    \Pd_{T_0,x;\mathcal{L}}\left( \Delta b_k(v) = 1 \mid \Delta \tau_k \text{ is good}, \mathcal{F}_{\tau_{k-1}}\right) = F_{k-1}(v)\;,
	$$
	which together with the fact vertices can increase their blue degree only at good times, $B_k \le 2k$ and  \eqref{eq:Forder} implies that 
	\begin{equation*}\label{eq:Fbk0}
	\Ed_{T_0,x;\mathcal{L}}\left[\sum_{v \in \mathcal{B}_{k-1}(0)} \mathbb{1}\left\{ \Delta b_{k}(v)=1\right \} \middle | \widetilde{\mathcal{F}}_{k-1}\right] = W_{k-1}\sum_{v \in \mathcal{B}_{k-1}(0)}F_{k-1}(v)  =  W_{k-1}\mathcal{O}\left(\frac{R_{k-1}}{k-1}\right).
	\end{equation*}
	Then,  using the above identity and \eqref{eq:prob_increments} we obtain 
	\begin{equation}\label{eq:average_increments}
	\begin{split}
	\Ed_{T_0,x;\mathcal{L}}\left[\Delta B_k(d) \mid \widetilde{\mathcal{F}}_{k-1}\right] &=
	W_{k-1}\, \frac{(d-1)B_{k-1}(d-1)- d\, B_{k-1}(d)}{2(k-1)}\\
	& \quad + W_{k-1}\left(Y_{k-1}(d-1) 
	- Y_{k-1}(d) \right)\;,
	\end{split}
	\end{equation}
	where, $Y_{k-1}(0)\equiv1$ and  for every $d\geq 1$,
	\begin{align}\label{eq:Yd}
	Y_{k-1}(d):=\begin{cases}
	\sum_{v \in \mathcal{B}_{k-1}(d)} F_{k-1}(v) - \sum_{v \in \mathcal{B}_{k-1}(0)} F_{k-1}(v)\;, & d =1\;, \\
	    \sum_{v \in \mathcal{B}_{k-1}(d)} F_{k-1}(v)\;, & d>1\;.
	    \end{cases}
	\end{align}
	Notice that for all $d\ge 1$, by \eqref{eq:Forder} we have that 
	$$
	    Y_{k-1}(d) = \mathcal{O}\left( \frac{R_{k-1}}{k-1}\right).
	$$
In order to simplify our notation, we now set $Q_{k}(d):= B_k(d)/k$. Then, writing 
	$$
	Q_k(d) - Q_{k-1}(d) = Q_k(d) - \Ed_{T_0,x; \mathcal{L}}\left[Q_k(d) \mid \widetilde{\mathcal{F}}_{k-1}\right] + \Ed_{T_0,x; \mathcal{L}}\left[Q_k(d) \mid  \widetilde{\mathcal{F}}_{k-1}\right] - Q_{k-1}(d)\;,
	$$
	and applying  \eqref{eq:average_increments}, we obtain the following recursion for~$Q_k(d)$:
	\begin{align}\label{eq:decomposition_lemma}
	Q_{k}(d) - Q_{k-1}(d) = \frac{1}{k} \left( \Psi_{k-1}(d-1) -Q_{k-1}(d) \Phi_{k-1}(d) \right) + \Delta M_{k-1}(d)\;,
	\end{align}
	where, for every $d\geq 1$, we have that 
	\begin{itemize}
		\item $\Psi_{k-1}(d-1)=\Big(\frac{d-1}{2}Q_{k-1}(d-1) + Y_{k-1}(d-1) - Y_{k-1}(d)\Big)
		W_{k-1}$;
		\item $\Phi_{k-1}(d)= 1 + \frac{d}{2}W_{k-1} = 1 + \frac{d}{2}\Pd_{T_0,x;\mathcal{L}}\left(\Delta \tau_k \text{ is good}\mid  \widetilde{\mathcal{F}}_{k-1}\right)$;
		\item $\Delta M_{k-1}(d) = Q_{k}(d) -  \Ed_{T_0,x; \mathcal{L}}\left[ Q_{k}(d) \mid  \widetilde{\mathcal{F}}_{k-1}\right]$.
	\end{itemize}
	In light of Lemma \ref{lem:approximation}, our goal is to guarantee that both of the processes~$\{\Psi_{n}(d-1)\}_{n\ge 0}$ and $\{\Phi_{n}(d)\}_{n\ge 0}$ have positive and deterministic limits and $\{M_n(d)\}_{n\geq 0}$ is convergent, for all fixed $d$.
	 
	The convergence of $\{\Phi_{n}(d)\}_{n\ge0}$ to $1+d/2$ (i.e. that $W_{k-1} \to 1$) comes as a consequence of Lemma~\ref{lemma:limtauk}. Indeed, observe that
	$$
	W_{k-1} = \Pd_{T_0,x;\mathcal{L}}\left(\Delta \tau_k \ge k^{2+\delta}+1\mid  \mathcal{F}_{\tau_{k-1}}\right) = \prod_{s=\tau_{k-1}+1}^{\tau_{k-1}+k^{2+\delta}+2}(1-p_s) \to 1\;,
	$$
	$\Pd_{T_0,x;\mathcal{L}}$-a.s. as $k$ goes to infinity, since $p_s = \Theta(s^{-\gamma})$ and by Lemma~\ref{lemma:limtauk}, $k^{2+\delta}/\tau_k^{\gamma}$ goes to zero almost surely.
	
	\medskip 
	
	\noindent \underline{Convergence of $\{M_n(d)\}_{n\geq 0}$:} For a fixed $d$, the terms $\Delta M_{k}(d)$ define a martingale difference sequence. Setting~$M_0(d)= 0$, let the process $\{M_n(d)\}_{n\geq 0}$ be the corresponding martingale. Diving both sides of  \eqref{eq:average_increments} yields to
	$$
	    \Ed_{T_0,x; \mathcal{L}}\left[ Q_{k}(d) \mid  \widetilde{\mathcal{F}}_{k-1}\right] = \frac{B_{k-1}(d)}{k} + \mathcal{O}\left(\frac{B_{k-1}(d) + B_{k-1}(d-1)+ R_{k-1}}{k^2 }\right),
	$$
	which implies that
	\begin{equation*}
	\begin{split}
	|\Delta M_{k-1}(d) | = \frac{\Delta B_k(d)}{k}+\mathcal{O}\left(\frac{B_{k-1}(d) + B_{k-1}(d-1)+ R_{k-1}}{k^2 }\right)\;.
	\end{split}
	\end{equation*}
	On the other hand, since $\Delta B_k(d) \le 2$ and $B_k(d)$ and $R_k$ are less than~$2k$ for all $d$, it follows that
	\begin{align*}
    (\Delta M_k(d))^2 = 
	\mathcal{O}\left(\frac{1}{k^2}\right),
	\end{align*}
	which is sufficient to guarantee that $\{M_n(d)\}_{n\geq 0}$ is bounded in $L^2$ and thus converges $\Pd_{T_0,x;\mathcal{L}}$-almost surely. \\
	%

	\noindent \underline{Convergence of $\{\Psi_{k}(d-1)\}_{k\ge 0}$:} This part requires induction on $d$, since $\Psi_{k}(d-1)$ is defined via $Q_{k}(d-1)$. Thus, our first step is to prove it for $d=1$.
	
	Observe that as a consequence of Lemma~\ref{lemma:Rk} and  \eqref{eq:Yd}, we have that $\{Y_n(d)\}_n$ converges to zero for $d\ge 1$. Moreover, when $d=1$, Lemma~\ref{lemma:limtauk} implies that 
	\begin{align*}
	\lim_{k \rightarrow \infty}\Psi_{k-1}(0)&=    \lim_{k \rightarrow \infty}(1 -Y_{k-1}(1)) \Pd_{T_0,x;\mathcal{L}}\left(\Delta \tau_k \text{ is good}\mid   \widetilde{\mathcal{F}}_{k-1}\right) = 1\;.
	\end{align*}
	Combining the above with our previous results and Lemma~\ref{lem:approximation} gives us that $\{Q_n(1)\}_{n\ge 0}$ converges to $\frac{1}{1+1/2}=2/3$.
	Now let $d\ge 2$. Using induction we show that $\lim_{k\to \infty}Q_{k}(d)$ exists a.s., and calculate the limit recursively. So let us assume  that $\lim_{k\to \infty}Q_k(d-1)=:Q(d-1)$ exists almost surely.  It follows that $$\lim_{k\to \infty}\Psi_k(d-1)=(d-1)Q(d-1)/2,\ \text{a.s.}$$  Using  Lemma~\ref{lem:approximation} again,  we conclude  that 
	\[
	Q(d):=\lim_{k\to \infty}Q_{k}(d) = \frac{\frac{d-1}{2}\cdot Q(d-1)}{(1+ d/2)} =\frac{d-1}{d+2}\cdot Q(d-1) \;, \quad \text{ $\Pd_{T_0,x;\mathcal{L}}$-a.s.}
	\]
	which implies that for $d\geq 1$,  $$\lim_{k\to \infty}Q_{k}(d)=\frac{4}{d(d+1)(d+2)},\quad\mathbb{P}_{T_0,x; \mathcal{L}}\text{-a.s.}$$ and proves Claim \ref{Bk}.
\end{claimproof}

\medskip
Returning now to the proof of Theorem~\ref{thm:powerlaw}, recall that, as observed at the beginning of our argument, the statement follows from  Claim \ref{Bk}, and we are done.
\end{proof}

\medskip 
Finally, having shown how Theorem \ref{thm:powerlaw} follows from our lemmas, we still owe the proofs of Lemmas \ref{lemma:limtauk}, \ref{lemma:Rk} and \ref{lemma:xtaucond}, and this is what we are going to do below.
\begin{proof}[Proof of Lemma~\ref{lemma:limtauk}]
Let us begin by noticing that the claim is true for $\gamma>1$, since in this regime one has $\tau_k=+\infty$, for all sufficiently large $k$, $\Pd_{T_0,x;\mathcal{L}}$-a.s. In the sequel, we therefore address the case $\gamma \in (2/3,1]$.   Let $V_m=|\vertexset{m}|$ and note that  $V_m := Z_1 +\dots +Z_m$, which implies that, as $m\to\infty$, \[
\Ed_{T_0, o;\mathcal{L}}(V_m) = \Theta (m^{1-\gamma})\;.
\]
Since $\tau_{k}$ is the first time of having $k+1$ vertices, if $T_0=\{o\}$, one has the identity
\begin{equation*}
\Lbrace \tau_{k} \le k^{1 /(1-\gamma) 
-\varepsilon
} \Rbrace = \Lbrace V_{k^{1 /(1-\gamma) 
-\varepsilon
}} \ge k\Rbrace\;,
\end{equation*}
 for any $\varepsilon < 1/(1-\gamma)$.
Recalling  the independence of the $Z_n$'s, using that $\log(1+x)<x,\ x>0$,  and bounding the sum by an integral, it follows that for every $\lambda>0$,
\begin{equation*}
\log\Ed_{T_0,x;\mathcal{L}} \left( e^{\lambda V_m}\right) \le \begin{cases}
 (e^{\lambda}-1)m^{1-\gamma}/(1-\gamma)\;, & \gamma <1\;,
\\
(e^{\lambda}-1)(\log m + 1) \;, & \gamma =1\;. 
\end{cases}
\end{equation*}
Thus, by the exponential Markov inequality  there is a positive constant $C$ depending on~$\gamma$ only, such that
\begin{equation}\label{two.exp}
\begin{split}
\Pd_{T_0,x;\mathcal{L}} \left( V_{k^{ 1/(1-\gamma) -\varepsilon}} \ge k \right)\le e^{-k}e^{Ck^{1-(1-\gamma)\varepsilon}}.
\end{split}
\end{equation}
This and the Borel-Cantelli lemma yields that $\tau_{k}<k^{1 /(1-\gamma) -\varepsilon}$ occurs only finitely many times, almost surely. We conclude the proof by choosing a small enough  $\varepsilon>0$ and noting that $\gamma > 2/3$ if and only if $\gamma/(1-\gamma) > 2$. 
\end{proof}
We now prove Lemma \ref{lemma:xtaucond}, while we defer the proof of Lemma \ref{lemma:Rk} to the end of this section.

\begin{proof}[Proof of Lemma \ref{lemma:xtaucond}]  
	%
	%
	%
	%
	Recall that $\{\mathcal{G}_n\}_{n\geq 1}$ denotes the natural filtration induced by the variables $\{Z_n\}_{n\geq 1}$. Observe that conditioned on $\mathcal{G}_{k-1}$, $\Delta \tau_k$ dominates a random variable $Y_k$ following geometric distribution of parameter ~$c/\tau_{k-1}^{\gamma}$, where $c$ is a positive constant depending on $(p_n)_n$ only.  Thus, with for $k$ sufficiently large, there exists a positive constant~$C$ such that, almost surely,
	\begin{equation}\label{eq:badtau}
	\begin{split}
	\Pd_{T_0,x;\mathcal{L}}\left( \Delta \tau_k \text{ is bad} \mid  \widetilde{\mathcal{F}}_{k-1} \right) & \le  \Pd_{T_0,x;\mathcal{L}}\left( Y_k < k^{2+\delta} +1 \mid  \widetilde{\mathcal{F}}_{k-1} \right)\le C\frac{k^{2+\delta}}{\tau_{k-1}^{\gamma}}\;.
	\end{split}
	\end{equation}
	Note that  Lemma \ref{lemma:limtauk} allows us to choose $\delta$ small enough so the RHS of the above inequality goes to zero almost surely when $k$ goes to infinity.
	
	Now, let $\Pi_k$ be the (random) stationary distribution of a simple random walk on the rooted tree $T_{\tau_k}$  with a self-loop attached to the root $o$ (we introduce a self-loop at the root just to avoid periodicity). More specifically, $\Pi_k$ is the random distribution over $\vertexset{\tau_k}$ which assigns to each $v\in \vertexset{\tau_k}$ the weight
	\begin{equation*}
	\Pi_k(v) := \frac{\degree{\tau_{k}}{v}}{2k}\;.
	\end{equation*}
	Let $P_{x, T}^t$ denote the distribution of a random walk over $T$ started at $x$ after $t$ steps. It is known (see \cite{brightwellwinkler1990}) that the mixing time of a random walk on a tree $T$ with $k$ vertices satisfies the  upper bound
	\begin{equation}\label{eq:boundtmix}
	\tmix \le  2k^2\;,
	\end{equation}
	uniformly in the initial state $x\in T$. Recall that $\rm d_{\rm TV}$  denotes the total variation distance. It is well known that for any $\varepsilon>0$ the  upper bound  
	\begin{equation*}
	\tmix(\varepsilon) \le \log\left(\frac{1}{\varepsilon}\right)\tmix \;,
	\end{equation*}
	holds, and
	in particular, for $\varepsilon := \exp\{-k^\delta\}$  by \eqref{eq:boundtmix}, we obtain  that  
	\begin{equation}\label{eq:dtvpt}
	\begin{split}
	\dtv{\Pi_{k-1}}{P^{t}_{X_{\tau_{k-1}}, T_{\tau_{k-1}}}} \le  e^{-k^\delta},\qquad \Pd_{T_0,x;\mathcal{L}}\text{-a.s.}
	\end{split}
	\end{equation}
	for all $t > k^{2+\delta}$. 
	Next, the strong Markov property applied to $\mathcal{L}$-\name at time $\tau_{k-1}$ yields that
	\begin{equation}\label{eq:markov}
	\Pd_{T_0,x;\mathcal{L}}\left( X_{\tau_{k} - 1} = v \mid \mathcal{F}_{\tau_{k-1}}\right) = \Pd_{X_{\tau_{k-1}}, T_{\tau_{k-1}}; \mathcal{L}^{(\tau_{k-1})}} \left( X_{\tau_1 -1} = v\right)\;.
	\end{equation} 

	Observe that under $\Pd_{X_{\tau_{k-1}}, T_{\tau_{k-1}}; \mathcal{L}^{(\tau_{k-1})}}$ conditioned on  $\tau_1 = t$, $X_{\tau_1 -1}$ has distribution $P^{t}_{X_{\tau_{k-1}}, T_{\tau_{k-1}}}$. The reader may check this fact recalling that the event $\tau_1 = t$ depends only on the random variables $\{Z_n\}_{n\geq 1}$ together with  the independence of $Z_n$ from $\mathcal{F}_{n-1}$, for all $n$. 
	
	Using  \eqref{eq:markov} and \eqref{eq:dtvpt} for $t>k^{2+\delta}$, we have that
	\begin{equation*}
	\Pd_{X_{\tau_{k-1}}, T_{\tau_{k-1}}; \mathcal{L}^{(\tau_{k-1})}}\left( X_{\tau_1 -1 } = v\mid \tau_1 = t\right) = \frac{\degree{\tau_{k-1}}{v}}{2(k-1)} \pm \mathsf{error}(v,k)\;,
	\end{equation*}
	where $\mathsf{error}(v,k)$ represents an $\mathcal{F}_{\tau_{k-1}}$-measurable  random variable for which 
	\[
	\Pd_{T_0,x;\mathcal{L}}\left(\mathsf{error}(v,k)\le e^{-k^\delta}\right)=1\;.
	\]
	This concludes the proof.
	\end{proof}

\begin{proof}[Proof of Lemma~\ref{lemma:Rk}] We know that  $\Delta R_k=2$  when $\Delta\tau_{k}$ is bad, while it equals $1$ when  $\Delta\tau_{k}$ is good and the half-edge $h_{k,2}$ is colored red. Therefore, 
	\begin{equation}\label{eq:expecdeltar}
	\begin{split}
	&\Ed_{T_0,x;\mathcal{L}} \left[\Delta R_k \mid
	\mathcal{F}_{\tau_{k-1}}
	\right]\\
	 &=  2\Pd_{T_0,x;\mathcal{L}} \left( \Delta\tau_{k}\text{ is bad}\mid 
	\mathcal{F}_{\tau_{k-1}}
	\right)+ \frac{R_{k-1}}{2(k-1)}\Pd_{T_0,x;\mathcal{L}} \left( \Delta\tau_{k}\text{ is good}\mid
	\mathcal{F}_{\tau_{k-1}}
	\right) \\
	& = \frac{R_{k-1}}{2(k-1)} + \mathsf{error}(k)\;,
	\end{split}
	\end{equation}
where, $\mathsf{error}(k)$ is an $\mathcal{F}_{\tau_{k-1}}$-measurable variable such that 
\begin{equation}\label{eq:bounderrobad}
 |\mathsf{error}(k)| \le 2\Pd_{T_0,x;\mathcal{L}} \left( \Delta\tau_{k}\text{ is bad}\mid
 \mathcal{F}_{\tau_{k-1}}
 \right), \; \Pd_{T_0,x;\mathcal{L}}\text{-a.s.}\;
\end{equation}
Now, define the function $\phi: \{2,3,,...\}\to (1,\infty)$ by 
\begin{equation*}
\phi (k):= \prod_{j=1}^{k-1}\left(1+\frac{1}{2j}\right) =2^{1-k}\frac{(2k-1)!!}{(k-1)!}= \Theta(\sqrt{k})\;,
\end{equation*}
where the last step uses Stirling's formula and that $(2k-1)!!=\frac{(2k)!}{2^{k}k!}$.
In fact, even $\phi (k)/\sqrt{k}\to \text{const}$, which can be computed.
Let 
$R'_k := \frac{R_k}{\phi(k)}$.
Using \eqref{eq:expecdeltar}, a simple computation shows that
\begin{equation*}
\Ed_{T_0,x;\mathcal{L}} \left[\Delta R'_k \mid
 \mathcal{F}_{\tau_{k-1}}
\right] = \frac{\mathsf{error}(k)}{\phi(k)}\;.
\end{equation*}
Next, exploiting Doob's decomposition, we may write $R'_k$ as
\begin{equation}\label{Doob.decomp}
R'_k = M_k + \sum_{m=1}^{k} \frac{\mathsf{error}(m)}{\phi(m)}\;,
\end{equation}
where the process $M$ is a mean zero martingale whose increments satisfy that
\begin{equation*}
|M_{j+1} - M_j | = \left | \frac{R_{j+1} -(1+\frac{1}{2j})R_j - \mathsf{error}(j+1)}{\phi(j+1)} \right | \le \frac{5}{\phi(j+1)}\;, \; \Pd_{T_0,x;\mathcal{L}}\text{-a.s.}
\end{equation*}
Thus, there exists a positive constant $C_1$ such that
\begin{equation}\label{the.bound}
\sum_{j=1}^k |M_{j+1} - M_j |^2 \le C_1\log k\;.
\end{equation}
Combining \eqref{the.bound} with the Azuma–Hoeffding inequality \cite{BBD2015Ineq} yields a positive constant $C_2$ depending on $\gamma$ and $C_1$ such that 
\begin{equation*}
\Pd_{T_0,x;\mathcal{L}} \left( M_k > C_2 k^{\delta'/2}\log^{1/2} k \right) \le e^{-k^{\delta'}}\;.
\end{equation*}
To control the predictable term in Doob's decomposition (i.e. the sum in \eqref{Doob.decomp}), we break the sum into two terms as $\sum_{m=1}^{k^{\delta'}}+\sum_{j=k^{\delta'}}^k$, by fixing some $\delta'$ to be chosen later. For the first sum we have the  deterministic bound
\begin{equation*}
\begin{split}
\sum_{m=1}^{k^{\delta'}}\frac{\mathsf{error}(m)}{\phi(m)} \le 2\sum_{m=1}^{k^{\delta'}} \frac{1}{\phi(m)} \le C_3k^{\delta'/2}\;,
\end{split}
\end{equation*}
for some universal constant $C_3$. For the rest of the terms, recall from the proof of Lemma~\ref{lemma:limtauk} (see \eqref{two.exp}) that for any $\varepsilon'<1/(1-\gamma)$ and for large enough $j$,
\begin{equation}\label{eq:boundtj}
\Pd_{T_0,x;\mathcal{L}} \left( \tau_{j} \le j^{1 /(1-\gamma) -\varepsilon'}\right) \le e^{-C_4j}\;.
\end{equation}
Thus, recalling  \eqref{eq:bounderrobad} and \eqref{eq:badtau}, on the event $\bigcap_{j=k^{\delta'}}^k \Lbrace \tau_{j} \ge j^{1 /(1-\gamma) -\varepsilon'} \Rbrace$ we have that
\begin{equation*}
\sum_{j=k^{\delta'}}^k \frac{\mathsf{error}(j)}{\phi(j)} \le \sum_{j=k^{\delta'}}^k C_5 j^{2+\delta - (\gamma/(1-\gamma) -\gamma\varepsilon') -1/2} \le C_6k^{1/2 -\varepsilon}\;,
\end{equation*}
where $\varepsilon>0$  depends on $\gamma$, which is fixed and on $\delta$ and $\varepsilon'$, which can be chosen properly. On the other hand, by the union bound and  \eqref{eq:boundtj} we obtain that 
$$
\Pd_{T_0,x;\mathcal{L}} \left( \bigcup_{j=k^{\delta'}}^k\Lbrace \tau_{j}
\le j^{1 /(1-\gamma) -\varepsilon'} \Rbrace\right) \le \exp\{-C_7k^{\delta'}\}\;.
$$
Since $R_k = R'_k\phi(k)$, it follows that, on the  event  $$\{ M_k \le C_2k^{\delta'/2}\log^{1/2} k\}\,\cap\,\bigcap_{j=k^{\delta'}}^k \Lbrace \tau_{j} \ge j^{1 /(1-\gamma) -\varepsilon'} \Rbrace\;,$$
 the estimate $R_k \le k^{1-\varepsilon}$ holds, which is enough to conclude the proof. 
\end{proof}

\section{Example: recurrence  without power-law degree distribution}
In order to illustrate that for the \name random walk recurrence is not equivalent to observing a scale-free tree sequence, we provide an example of a sequence $\{Z_n\}_{n\ge 1}$ under which the walker is recurrent but the corresponding tree sequence has a limiting degree distribution which is very different from the power-law. In fact, we prove that in this setup, the proportion of leaves converges to one almost surely.

To carry out this program, for a fixed $\delta >0$, consider the sequence of real numbers $(p_n^\delta)_{n \ge 1}$, with $p_n^{(\delta)} := {n^{-\delta}}$, 
and the laws 
\begin{equation}\label{def:zn}
L^{(\delta)}_n:= 
\begin{cases}
\log n,\text{ with probability }p_n^{(\delta)};\\
0,\text{ with probability }1-p_n^{(\delta)}.
\end{cases}
\end{equation}
Regarding the $\mathcal{L}^{(\delta)}$-\name process  we have the following result:
\begin{proposition}[Recurrence with leaves only]\label{prop:deltadist} For the $\mathcal{L}^{(\delta)}$-\name process satisfying \eqref{def:zn}, the random walk is recurrent for all $\delta \in (2/3,1]$. Furthermore, if $N_n(1)$  denotes the number of leaves and $|\mathcal{V}_n|$ the total number of vertices in $T_n$, then
$$
    \lim_{n\to \infty } \frac{N_n(1)}{|\mathcal{V}_n|} = 1\;, \quad  \mathbb{P}_{T_0,x;\mathcal{L}^{(\delta)}}\text{-a.s.}
$$
\end{proposition}
\underline{Intuition:} By \eqref{def:zn}, when the walker adds new leaves at time $n$, it adds an amount of $\log n$. Heuristically, Proposition~\ref{prop:deltadist} says that this quantity is not large enough to ``trap'' the walker in some neighborhood of a vertex, so the walk is still recurrent. Nonetheless, it is large enough to change the power-law degree distribution dramatically, producing trees whose empirical degree distribution converges to the degenerate distribution over $\mathbb{N}$ which puts all the mass at $1$.
\begin{proof}[Proof of Proposition \ref{prop:deltadist}]
We check recurrence first.

\underline{Recurrence of the walk.}  By Theorem \ref{thm:general_recurrence}, we need to show that $\{Z^{(\delta)}_n\}_n$ satisfies all the assumptions (A1)-(A3) of that theorem. Assumptions (A1) and (A2) trivially follow  from the definitions of $p_n^{(\delta)}$ and $Z^{(\delta)}_n$. For (A3), recalling that $\delta \in (2/3,1]$, use that
$
    E(Z^{(\delta)}_n) = \frac{\log n}{n^{\delta}},$
    yielding $$M_n = \sum_{k=1}^n \frac{\log k}{k^{\delta}} \le  \log n \sum_{k=1}^n \frac{1}{k^{\delta}}\le \log n\int_0^n x^{-\delta}\, dx=(1-\delta)^{-1}\log n \cdot n^{1-\delta} \;,
$$
hence $$(1-q_n)\cdot M_n^2\le {n^{-\delta}}\cdot(1-\delta)^{-2} n^{2-2\delta}(\log n)^2=
 (1-\delta)^{-2}n^{2-3\delta}(\log n)^2\to 0\;,$$ since $\delta>2/3.$ This proves recurrence. \\

Next we check the degree distribution.

\underline{Proportion of leaves.} Define the stopping times
\begin{equation}\label{eq:stopping.t}
     \tau_0\equiv 0;\qquad  \tau_k: = \inf \{ n > \tau_{k-1} \; : \; Z^{(\delta)}_n = \log n \}\;,\qquad k\ge 1\;.
\end{equation}
 Since $\delta \le 1$, it follows that all the $\tau_k$ are finite $P$-a.s. Introduce the shorthands 
$$
\widetilde{N}_k := N_{\tau_k}(1); \quad \widetilde{V}_k := |\mathcal{V}_{\tau_k}|\;.
$$
By \eqref{eq:stopping.t}, we have that
$
\widetilde{V}_k = |\mathcal{V}_0| + \sum_{j=1}^k \log(\tau_k)
$ where $|\mathcal{V}_0|$ is size of the initial tree $T_0$. Moreover, observe that whenever a certain amount of new leaves is added, even in the ``worst case scenario'' when new vertices are added to a leaf, only one leaf has to be subtracted from the total number of leaves. This implies that 
\begin{equation*}
    \widetilde{N}_k \ge \sum_{j=1}^k\log(\tau_j) -k = \widetilde{V}_k - |\mathcal{V}_0| - k\;,  \quad
    \mathbb{P}_{T_0,x;\mathcal{L}^{(\delta)}}\text{-a.s.}
\end{equation*}
Using that $\tau_k \ge k$ almost surely, we have that 
\[
\widetilde{V}_k \ge \sum_{j=1}^k \log j \ge \int_1^k \log x\, dx= k(\log k-1)+1\;,
\]
yielding that
\begin{align*}
 &\frac{\widetilde{N}_k}{\widetilde{V}_k} \ge 1 - \frac{ (|\mathcal{V}_0| + k)}{\widetilde{V}_k}\ge 1- \frac{ (|\mathcal{V}_0| + k)}{k(\log k-1)+1}\\
 &\implies \liminf_{k \to \infty} \frac{\widetilde{N}_k}{\widetilde{V}_k} \ge 1\;,
 \quad
 \mathbb{P}_{T_0,x;\mathcal{L}^{(\delta)}}\text{-a.s.}\;
\end{align*}
To conclude the proof use that $\widetilde{N}_k \le \widetilde{V}_k$ along with the obvious fact that the sequence $\{N_n(1)/|\mathcal{V}_n|\}_n$ only changes its values at the stopping times $\tau_k$'s.
\end{proof}
\section{Transience when many edges are grown}\label{tr.many}
The purpose of this section is to demonstrate that when there are infinitely many growth times (i.e. $p_n$ is not summable),
and there are sufficiently many edges  grown at those times, the walk is never recurrent, and under a mild condition on $p_n$ (ruling out that the walk ``gets stuck'') it is transient.

Let $Z_n=w_n$ with probability $p_n$ and $Z_n=0$ otherwise. Here $\{w_n\}_{n\ge 1}$ is a non-decreasing numerical sequence such that $w_n\in \mathbb N\setminus \{0\}$.

We will denote  $\mathfrak P_n:=\sum_{k=1}^n p_k$,  
and for a sequence $\{a_n\}_{n\geq 1}$, we will write $\Delta a_i:=a_i-a_{i-1}$.
The graph distance between the root $o$ and vertex $v$ will be denoted by $d(o,v)$.
Finally, denote 
$\tau_n := \inf \Lbrace n > \tau_{n-1} \, : \, Z_n >0 \Rbrace\;$, with $\tau_0 \equiv 0$, and  $r_{i,j}:=\Pd_{T_0,o; \mathcal{L}}(\tau_i>j),i,j\in \mathbb N$;  note that $r_{i,j}$ is a function of the sequence  $\{p_n\}_{n\geq 1}$ only. 
\begin{remark}[Poisson-Binomial]
The $r_{i,j}$ can be computed for a fixed sequence of $p_n$'s using the identities of the Poisson-Binomial distribution, since the number of growths up to time $m$ follows a Poisson-Binomial distribution with parameters $m,p_1,...,p_m$ and mean $\mathfrak P_m$, for $m\ge 1$.
If this number is $K^{m,p_1,...,p_{m}}$ then 
$$r_{i,j}=P(K^{j,p_1,...,p_{j}}\le i-1)\;.$$
Since we only need an upper bound in the next result, we can simply use Chebyshev.  Indeed, if $Y_i:= \mathbb{1}\{Z_i >0\}$, i.e.,  the indicator of the $i$-th growth then $r_{i,j}= P(Y_1+...+Y_j\le i-1)$, and by the Chebyshev inequality,
\begin{equation}\label{Ch.bound}
  r_{i,j}\le  \frac{\sum_{k=1}^j p_k(1-p_{k})}{\left(\mathfrak P_j-i+1\right)^2}\le 
\frac{\mathfrak P_j}{\left(\mathfrak P_j-i+1\right)^2}\;,
\end{equation}
 provided $\mathfrak P_j>i-1.\hfill\diamond$
\end{remark}
\begin{theorem}[Transience]\label{thm:tr}
Assume that 
$\sum\limits_{n\ge 1}p_n=\infty$ to rule out obvious recurrence, and that there exists a numerical sequence $0<a_1<a_2,\dots$ such that
\begin{enumerate}
    \item $$\sum_{i=1}^{\infty}r_{i,a_{i}}<\infty;$$
\item $$\sum_{i=1}^{\infty}\frac{\Delta a_i}{w_{i-1}+1}<\infty.$$
\noindent Then the $\mathcal{L}$-\name is not recurrent.

If furthermore we assume that
\item \begin{equation}\label{mild.condition}
\sum_{n\ge 1} \min\{p_n,p_{n+1}\}=\infty.
\end{equation}
\end{enumerate}
then the $\mathcal{L}$-\name is transient: $\Pd_{T_0,o; \mathcal{L}}\left(\lim\limits_{n\to \infty} d(o,X_n)=\infty\right)=1$.
\end{theorem}
\begin{corollary}
For every non-summable sequence $\{p_n\}_{n\geq 1}$ there exists some sequence $\{w_n\}_{n\geq 1}$, making the \name{}  non-recurrent, and under Condition 3 of Theorem \ref{thm:tr} even transient.
\end{corollary}
\begin{example}[Transience]
Let $p_i:=\frac{1}{i+1}$, hence $\mathfrak P_n=\mathcal{O}(\log n)$. Then for any $\{w_i\}$ growing so fast that with some $\epsilon>0$
$$\sum_{i=1}^{\infty}\frac{2^{i^{1+\epsilon}}}{w_{i-1}+1}<\infty\;,$$
holds,
the $\mathcal{L}$-\name is transient. To verify this, pick $a_i:=2^{i^{1+\epsilon}},\ \epsilon>0.$ Then
by \eqref{Ch.bound}, $$r_{i,a_i}\le \frac{\mathfrak P_{a_{i}}}{(\mathfrak P_{a_{i}}-i+1)^2}\;,$$
which is $\mathcal{O}(i^{-1-\epsilon}),$ thus summable.$\hfill\diamond$
\end{example}
The next example shows that the condition \eqref{mild.condition} is indeed essential for transience.
\begin{example}[The walk is not recurrent and the walker is ``stuck'']
Consider the sequence of $p_n$'s $$1/2,1/2,1/3,1/4,1/4,1/8,1/5,1/16,...$$
which is obtained by ``combing'' the harmonic series $1/2,1/3,1/4,...$ and the convergent series 
$1/2,1/4,1/8,...$ together; condition \eqref{mild.condition} is not satisfied in this case.
A close look at the proof of Theorem \ref{thm:tr} reveals that for this choice of $p_n$ and sufficiently large $w_n$ (so that the first two conditions of Theorem \ref{thm:tr} are met) with positive probability $d(o,X_n)$ can be stuck between two positive integers forever.$\hfill\diamond$
\end{example}
\begin{proof}[Proof of Theorem \ref{thm:tr}]
We prove the result in three steps.

\medskip
\underline{STEP 1:} We now assume that \eqref{mild.condition} holds. Recall that $\tau_i$ is the $i$-th time of growth, $i\ge 1,$ 
and for $k\ge 1$, let
$$v_k:=\text{\ the ``father vertex'' launching $w_k$ new edges at}\ \tau_k;$$
$$A_k:=\{\exists i\ge 1: d(o, X_n)> d(o, v_k) \; \forall n\geq \tau_k +i\}\;.$$
Note the following. Writing simply $d(v)$ instead of $d(o,v)$, on the event $\limsup_k A_k,$ the walker
$X$ eventually leaves all closed (!) balls of radius $d(v_k)$ around $o$. In particular then, $d(v_{k+1})>d(v_k)$. Hence, in fact $X$ eventually leaves arbitrarily large closed balls, and this implies transience.
So, if we show that 
\begin{equation}\label{limsup=1}
    \Pd_{T_0,o; \mathcal{L}}\left(\limsup_{k\to\infty} A_k\right)=1\;,
\end{equation} then transience follows.

We start by showing that for any $\epsilon>0$ there is a deterministic integer $k_0=k_0(\epsilon)\ge 1$  such that if $k\ge k_0$ then  
\begin{align}\label{close.to.1}
\pi_k:=&P(A_k)\ge 1-\epsilon\;.
\end{align}
To achieve this,  first note that
\begin{equation}\label{union.bound}
\Pd_{T_0,o; \mathcal{L}}(\cap_{i=k}^{\infty}\{\tau_i\le a_i\})= 1-\Pd_{T_0,o; \mathcal{L}}(\cup_{i=k}^{\infty}\{\tau_i>a_i\})\ge 1-\sum_{i=k}^{\infty}r_{i,a_{i}}\;.
\end{equation}

\underline{Heuristic explanation for estimating $\pi_k$ from below:}
First note that clearly, $\tau_k\ge k$, thus $w_{\tau_k}\ge w_k$, given our monotonicity assumption on the $w_i$'s.
Consider the following strategy.
For each $k\ge 0$, on the time interval $[\tau_k,\tau_{k+1}-1]$ the walker chooses to increase her distance from $o$ every time she is at a vertex of degree $1+w_k$ or larger; each time this happens with probability  at least $w_k/(w_k+1)$.

We now claim that this strategy guarantees that $A_k$ occurs almost surely, that is, that the walker, while implementing this strategy, eventually leaves the ball of radius $d(v_k)$, whatever $k\ge 1$ is.

Indeed, for any $k\ge 1$, the event $v_k=v_{k+1}=v_{k+2}=...$ has probability zero. To see this recall \eqref{mild.condition} and note that at least one time out of any two consecutive times, the walker is at a vertex different from $v_k$. 
This means that the walker eventually stays at a distance from the origin which is not just at least $d(v_k)$ but also at least $d(v_{k+i})>d(v_k)$, with some $i\ge 1$. 

Since the above strategy guarantees that $A_k$ occurs a.s., by \eqref{union.bound} one has that
\begin{equation}\label{lower.bd.pik}
   \pi_k\ge \left(1-\sum_{i=k}^{\infty} r_{i,a_i}\right)\cdot\prod_{i=k}^{\infty}\left( \frac{w_{i-1}}{w_{i-1}+1}\right)^{\Delta a_i}, \quad k\ge 1\;.
\end{equation}

Using that for some $c>e$, 
$$\left( \frac{w_{i-1}}{w_{i-1}+1}\right)^{\Delta a_i}> c^{- \frac{\Delta a_{i}}{w_{i-1}+1}},\qquad i\ge 2,$$
it follows that
\begin{equation}\label{sum.product}\prod_{i=k}^{\infty}\left( \frac{w_{i-1}}{w_{i-1}+1}\right)^{\Delta a_i}\ge
\prod_{i=k}^{\infty}c^{- \frac{\Delta a_{i}}{w_{i-1}+1}}=c^{-\sum_{i=k}^{\infty} \frac{\Delta a_{i}}{w_{i-1}+1}}\;.
\end{equation}
Exploiting \eqref{lower.bd.pik} and \eqref{sum.product}  the bound $\pi_k>1-\epsilon$ follows if, with some appropriately small $\delta(\epsilon)>0$,
\begin{equation*}
\sum_{i=k}^{\infty} r_{i,a_i}<\delta(\epsilon)\text{\qquad and\qquad}
\sum_{i=k}^{\infty}\frac{\Delta a_i}{w_{i-1}+1}<\delta(\epsilon)\;,
\end{equation*}
which, in turn, are guaranteed for large enough $k$'s by our summability assumptions.

\medskip
\underline{STEP 2:}
So far, we have shown that, assuming  \eqref{mild.condition},  for some (deterministic) integer $k_0(\epsilon)\ge 1$,
$$\Pd_{T_0,o; \mathcal{L}}(A_k)\ge 1-\epsilon,\qquad \forall k\ge k_0(\epsilon).$$
By the reverse Fatou inequality then, $$\Pd_{T_0,o; \mathcal{L}}\left(\limsup_{k\to\infty} A_k\right)=\Pd_{T_0,o; \mathcal{L}}\left(\limsup_{i\to\infty} A_{k_{0}+i}\right)\ge 1-\epsilon.$$
Since $\epsilon>0$ was arbitrary, \eqref{limsup=1} follows, and hence we are done with showing transience under conditions 1-3.

\medskip
\underline{STEP 3:} To complete the proof of the theorem, a careful look at the first two steps reveals that if we drop \eqref{mild.condition}, then although the argument for the walker going out to infinity breaks down, it still follows that she will be bounded away from the root a.s., hence the process is not recurrent.
\end{proof}

%


\medskip 
\paragraph{\bf Acknowledgments.}{J. E. is indebted to Y. Peres for first bringing the model to his attention and for further helpful conversations.}


\bibliographystyle{amsplain}
\bibliography{ref2023.bib}












\end{document}